\documentclass[11pt]{amsart}
\usepackage[english]{babel}
\usepackage[dvips]{graphicx}
\usepackage{amsmath,amsfonts,amsthm,amssymb,bbm}
\usepackage[usenames, dvipsnames]{color}

\newtheorem{theo}{Theorem}[section]

\newtheorem{ass}{Assumption}[section]
\newtheorem{lem}{Lemma}[section]
\newtheorem{cor}{Corollary}[section]
\newtheorem{rem}{Remark}[section]

\title[Synchronization of RSP with a
  network-based interaction]{\bf Synchronization of Reinforced Stochastic
  Processes with a Network-based Interaction}

\author[G. Aletti]{Giacomo Aletti}
\address{ADAMSS Center \& Department of Mathematics,
  Universit\`a degli Studi di Milano, 
via Saldini 50, 20133 Milan, Italy}
\email{giacomo.aletti@unimi.it}

\author[I. Crimaldi]{Irene Crimaldi}
\address{IMT School for Advanced Studies,
Piazza San Ponziano 6, 55100 Lucca, Italy}
\email{irene.crimaldi@imtlucca.it}

\author[A. Ghiglietti]{Andrea Ghiglietti}
\address{Department of
  Mathematics, Universit\`a degli Studi di Milano,
via Saldini 50, 20133 Milan, Italy}
\email{andrea.ghiglietti@unimi.it (Corresponding author)}

\date{\today}

\begin{document}

\maketitle

\begin{abstract}
Randomly evolving systems composed by elements which interact among
each other have always been of great interest in several scientific
fields. This work deals with the {\em synchronization} phenomenon,
that could be roughly defined as the tendency of different components
to adopt a common behavior. We continue the study of a model of
interacting {\em stochastic processes with reinforcement}, that
recently has been introduced in~\cite{cri-dai-lou-min}.  Generally
speaking, by reinforcement we mean any mechanism for which the
probability that a given event occurs has an increasing dependence on
the number of times that events of the same type occurred in the
past. The particularity of systems of such stochastic processes is
that synchronization is induced {\em along time} by the reinforcement
mechanism itself and does not require a large-scale limit. We focus on
the relationship between the topology of the network of the
interactions and the long-time synchronization phenomenon. After
proving the almost sure synchronization, we provide some CLTs in the
sense of stable convergence that establish the convergence rates and
the asymptotic distributions for both convergence to the common limit
and synchronization. The obtained results lead to the construction of
asymptotic confidence intervals for the limit random variable and of
statistical tests to make {\em inference on the topology of the
  network} given the observation of the reinforced stochastic
processes positioned at the vertices.
\end{abstract}

\\

\smallskip
\noindent {\em 2010 AMS classification:} 60F05, 60F15, 60K35, 62P35, 91D30.

\section{Introduction}\label{introduction}

The stochastic evolution of systems composed by elements which
interact among each other has always been of great interest in several
scientific fields: in neuroscience the brain is an active network
where billions of neurons interact in various ways in the cellular
circuits; many studies in biology focus on the interactions between
different sub-systems; social sciences and economics deal with
individuals that take decisions under the influence of other
individuals, and also in engineering and computer science some
research questions regard dynamic agents that form a complex pattern
of interactions (e.~g. \cite{bar-alb, hof, new}). In all these
frameworks, an usual phenomenon is the {\em synchronization}, that
could be roughly defined as the tendency of different components to
adopt a common behavior (we refer to \cite{are} for a detailed and
well structured survey on this topic, rich of examples and
references).  Synchronization has been shown to be of special
relevance in neural systems: the study of synchronization in neuronal
networks of various level, especially dealing with the role played by the
network topology, is crucial for the understanding of the brain
functional activities. In social life, preferences and beliefs are
partly transmitted by means of various forms of social interaction and
opinions are driven by the tendency of individuals to become more
similar when they interact. Hence, a collective phenomenon reflects
the result of the interactions among different individuals. The
underlying idea is that individuals have opinions that change
according to the influence of other individuals giving rise to a sort
of collective behavior, sometimes grouping together a part of the
whole population with similar social attributes. Moreover, in
economics, simple rules lead to interesting collective behaviors and
synchronization is one of them, since some of the activities done by
individual agents can become correlated in time due to their
interaction pattern. For example, the analysis of the International
Trade Network (ITN), also known as the World Trade Web (WTW), which is
the network related to the trade volume between countries, has
revealed a tight relationship between the topology of the ITN and the
dynamics of the Gross Domestic Products (GDPs) of the countries.  Due
to globalization effects, all economies are strongly correlated and
they will tend to follow a common trend (what are usually called
economic cycle). We can say that economies are synchronized in terms
of the GDP. Finally, consensus problems, understood as the ability of
a set of interacting dynamic agents to reach a unique and common value
in an asymptotic stable state, play a crucial role also in engineering
and computer science, particularly in automata theory. Therefore, it
is clear that the main goals of different research areas are twice:
(i) to figure out whether and when a (complete or partial)
synchronization in a dynamical system of agents can emerge out of
initially different statuses and (ii) to understand the interplay
between the network topology of the interactions among the agents and
the dynamics followed by the agents.  \\

\indent In this paper we continue the study of synchronization for a
model of interacting stochastic processes with reinforcement, that
recently has been introduced in \cite{cri-dai-lou-min}. Generally
speaking, by reinforcement in a stochastic dynamic we mean any
mechanism for which the probability that a given event occurs has an
increasing dependence on the number of times that events of the same
type occurred in the past. The main reason of the attention devoted to
reinforced stochastic processes is concerned with their dynamics,
which is suitable to describe random phenomena in different scientific
areas and can be easily implemented in several fields of application
(see e.g. \cite{pem} for a general survey). Our study is motivated by
the attempt of understanding the role of the reinforcement mechanism
in synchronization phenomena.\\

\indent More precisely, a {\em Reinforced Stochastic Process} (RSP)
can be defined as a stochastic process in which, along the time-steps,
different events occur in such a way that, for each event, the greater
the probability of occurence at a certain time, the greater the
probability of occurence at the next time.  Formally, given a finite
set $S$, for any $x$ in $S$, we have a stochastic process
$X(x)=(X_n(x))_{n\geq 1}$ with values in $\{0,1\}$ such that, for each
$n\geq 0$, $\sum_{x\in S} X_n(x)=1$ and
\begin{equation}\label{reinforced-1}
P(X_{n+1}(x)=1\, |\, Z_0(x),\, X_1(x),....,X_n(x))=Z_n(x)
\end{equation}
where
\begin{equation}\label{reinforced-2}
Z_{n}(x)=(1-r_{n-1})Z_{n-1}(x)+r_{n-1}X_{n}(x)
\end{equation}
with $0\leq r_{n-1}<1$ and $Z_0(x)$ possibly random. Indeed, the process
$X(x)$ describes the sequence of occurences of the ``event'' $x\in S$
and, if at time $n$, the ``event'' $x$ has taken place, then the
probability of its occurence at time $(n+1)$ increases.  Therefore, the
larger $Z_{n-1}(x)$, the higher the probability of having $Z_n(x)$
greater than $Z_{n-1}(x)$.  This {\em ``self-reinforcing property''},
also known as {\em ``preferential attachment rule''}, is a key feature
governing the dynamics of many biological, economic and social systems
(e.~g. \cite{pem}).  The prototype of reinforced stochastic processes
is the standard Eggenberger-P\'oya urn \cite{egg-pol, mah}: an urn
contains $a$ red and $b$ white balls and, at each discrete time, a
ball is drawn out from the urn and then it is put again inside the urn
together with one additional ball (or, more generally, with an
additional constant number of balls) of the same color. In this case,
we have $S=\{0,1\}$ with $1$ representing the color red and $0$ the
color white and for $Z_n=Z_n(1)$ and $X_n=X_n(1)$, we have
\begin{equation}
Z_n=\frac{a+\sum_{m=1}^{n}X_m}{a+b+n}.
\end{equation}
It is immediate to verify that
$$
Z_{0}=\frac{a}{a+b}\quad\mbox{and}\quad Z_{n+1}=(1-r_n)Z_{n}+r_nX_{n+1}
$$ with $r_n=(a+b+n+1)^{-1}$. As shown in \cite{cri-dai-lou-min} and as
we will see in this paper, the asymptotic behavior of $r_n$ is
essential to determine the results presented in this paper.  To this
purpose, here we highlight that for the Eggenberger-P\'oya urn we have
$\lim_n n r_n=1$.  We refer to \cite[Example 1.2]{cri-dai-lou-min}
for a meaningful case of reinforced stochastic process of the type
\eqref{reinforced-1}-\eqref{reinforced-2} where $\lim_n n^{\gamma}
r_n= c$ with $\gamma<1$ and $c\in (0,+\infty)$.  This example concerns an
opinion dynamics in an evolving population, modeled by a graph
evolving according to preferential attachment \cite{alb-bar, hof,
  new}.  \\

\indent To avoid complications, from now on we assume $S = \{0, 1\}$,
so that there is only one relevant variable $Z_n$, since $Z_n(1) =Z_n$
and $Z_n(2)=(1-Z_n)$.  This paper deals with a system of $N$
reinforced stochastic processes that interact according to a given set
of relationships among them. More precisely, suppose to have a
directed graph $G=(V, E)$ with $V=\{1,...,N\}$ as the set of vertices
and $E\underline{\subset} V\times V$ as the set of edges.  Each edge
$(j,k)\in E$ represents the fact that the vertex $j$ has a direct
influence on the vertex $k$. We assume also to associate a weight
$w_{j,k}\geq 0$ to each edge in order to quantify how much $j$ can
influence $k$.  A weight equal to zero means that the edge is not
present. We set $W=[w_{j,k}]_{j,k\in V\times V}$ ({\em weighted
  adjacency matrix}) and we assume the weights to be normalized so
that $\sum_{j=1}^N w_{j,k}=1$.  Hence, $w_{k,k}$ represents how much
the vertex $k$ is influenced by itself and $\sum_{j=1,j\neq k}^N
w_{j,k}\in[0,1]$ quantify how much the vertex $k$ is influenced by the
other vertices of the graph.  Finally, we suppose to have at each
vertex $j$ a reinforced stochastic process described by
$X^j=(X_{n,j})_{n\geq 0}$ and $Z^j=(Z_{n,j})_{n\geq 0}$ such that, for
each $n\geq 0$, the random variables $\{X_{n,j}:\,j=1,\dots, N\}$ are
conditional independent given ${\mathcal F}_{n-1}$ with
\begin{equation}\label{interacting-1-intro}
P(X_{n+1,j}=1\, |\, {\mathcal F}_n)=\sum_{k=1}^N w_{k,j} Z_{n,k}
\end{equation}
where, for each $k\in V$,
\begin{equation}\label{interacting-2-intro}
Z_{n,k}=(1-r_{n-1})Z_{n-1,k}+r_{n-1}X_{n,k}
\end{equation}
with $0\leq r_n<1$ and
${\mathcal F}_n=\sigma(Z_{0,k}: k\in V)\vee
\sigma(X_{m,j}: j\in V,\, m\leq n )$.\\

\indent As an example, we can imagine that $G=(V, E)$ represents a
network of $N$ individuals that at each time-step have to make a
choice between two possible alternatives $\{0,1\}$. We can formalize
this setting, assuming to have at each vertex $j$ an urn with red and
white balls. The color red represents the choice 1, the proportion
$Z_{n,j}$ of red balls at time $n$ in the urn at vertex $j$ represents
the inclination of the individual $j$ to adopt the choice 1 at time
$n$ and the random variable $X_{n,j}$ represents the choice of $j$ at
time $n$. It is natural to assume a self-reinforcing property for the
own inclination of each individual as in \eqref{interacting-2-intro}
and, moreover, it is natural to assume that the probability that the
individual $j$ will make the choice 1 at time $(n+1)$ is given by a
convex combination of $j$'s own inclination and the inclination of the
vertices that have an influence on $j$ according to their weights
$w_{k,j}$ as in \eqref{interacting-1-intro}. Another example is given
by the interacting version of \cite[Example 1.2]{cri-dai-lou-min},
which could be interpreted as a network of
different interacting populations or groups within a given population
in the same spirit as in \cite{contucci}.\\

\indent As already said at the beginning, our study deals with the
synchronization phenomenon of the stochastic processes $\{ (Z_{n,j})_n
:\, 1\leq j\leq N \}$ positioned at the vertices. The particularity of
systems of interacting reinforced stochastic processes is that {\em
  synchronization is induced along time} by the reinforcement
mechanism itself (independently of the fixed size $N$ of the network),
and so it does not require a large-scale limit (i.e. the limit for
$N\to +\infty$), which is usual in statistical mechanics for the study
of interacting particle systems. In particular, we focus on the
relationship between the topology of the interactions and the
long-time synchronization phenomenon: indeed, we show that the
eigenvalues and eigenvectors of the weighted adjacency matrix $W$
impact on the synchronization phenomenon. Our theoretical results
provide the rates of synchronization and the second-order asymptotic
distributions, in which the asymptotic variances have been expressed
as functions of the parameters governing the reinforced dynamics and
the eigen-structure of the weighted adjacency matrix. These results
lead to the construction of asymptotic confidence intervals for the
common limit random variable of the processes $(Z_{n,j})_n$ and to the
design of statistical tests to make inference on the topology of the
interaction network given the observation of the processes
$(Z_{n,j})_n$.
\\

\indent Regarding the literature review, we recall that interacting
two-colors urns have been considered in \cite{lau1, lau2}. Their main
results are proven when the probability of drawing a ball of a certain
color is proportional to $\rho^k$, where $\rho>1$ and $k$ is the
number of balls of this color. The interaction is of the mean-field
type. More precisely, the interacting reinforcement mechanism is the
following: at each step and for each urn draw a ball from either all
the urns combined with probability $p$, or from the urn alone with
probability $1-p$, and add a new ball of the same color to the
urn. The higher the interacting parameter $p$, the more memory is
shared between the urns. The main results can be informally stated as
follows: if $p\geq 1/2$, then all the urns fixate on the same color
after a finite time, and if $p<1/2$, then some urns fixate on a unique
color and others keep drawing both colors.  In \cite{cri-dai-min,
  dai-lou-min, sah} the authors consider interacting urns (precisely,
\cite{cri-dai-min} and \cite{dai-lou-min} deal with P\'olya urns and
\cite{sah} regards Friedman urns) in which the interaction can be
defined again as of the mean-field type, but the reinforcement scheme
is different from the previous one: indeed, the urns interact among
each other through the average composition in the entire system, tuned
by the interaction parameter $\alpha$, and the probability of drawing
a ball of a certain color is proportional to the number of balls of
that color, rather than to its exponential, leading to quite different
results. Synchronization and central limit theorems for the urn
proportions have been proven for different values of the tuning
parameter $\alpha$, providing different convergence rates and
asymptotic variances. In \cite{cri-dai-lou-min} the same mean-field
interaction is adopted, but the analysis has been extended to the
general class of reinforced stochastic processes, providing central
limit theorems also in functional form.  Differently from these works,
the model proposed in \cite{ale-ghi} concerns with a system of
generalized Friedman urns with irreducible mean replacement matrices
based on a general interaction structure, which includes the
mean-field interaction as a special case.  In particular, this
interaction acts as follows: the probability to sample a certain color
in each urn is a convex combination of the urn proportions of the
entire system, and the weights of such combinations are gathered in
the interacting matrix.  Combining the information provided by the
mean replacement matrices and by the interacting matrix, first and
second-order asymptotic results of the urn proportions have been
established, from which synchronization phenomenon has not been
observed.  Moreover, the structure of the interacting matrix allows a
decomposition in sub-systems of urns evolving with different
behaviors.  \\ \indent The present work have some issues in common
with \cite{cri-dai-lou-min, cri-dai-min} and \cite{ale-ghi}, but at
the same time some significant differences can be pointed out. In
particular, we share with \cite{ale-ghi} a general interacting
framework driven by the interacting matrix (here called weighted
adjacency matrix). However, here we mainly consider irreducible
interacting matrices and hence the decomposition of the system in
sub-groups is only sketched. Moreover, with respect to \cite{ale-ghi},
we study a class of stochastic processes for which we obtain
synchronization. This class does not include the generalized Friedman
urns studied in \cite{ale-ghi}; while it includes urn models with
not-irreducible mean replacement matrices, as P\'olya urns. With
\cite{cri-dai-lou-min} we share the class of reinforced stochastic
processes considered, which contain P\'olya urns also studied in
\cite{cri-dai-min}. However, with respect to \cite{cri-dai-lou-min,
  cri-dai-min}, we generalize the form of interaction since here we
deal with a general weighted adjacency matrix instead of just the
mean-field interaction. Indeed, the intent of this work is different
from the one of the above papers: after proving synchronization and
central limit theorems for some interesting cases, we focus on
analyzing the interplay between the topology of the interaction
network and the reinforced dynamics of the stochastic processes
positioned at the vertices of the network, providing some statistical
tools. On the other hand, we do not provide central limit theorems in
functional form as in \cite{cri-dai-lou-min} (although it is possible
to do it combining the results given here and the methods illustrated
in \cite{cri-dai-lou-min}) and we do not cover some cases considered
in \cite{cri-dai-lou-min, cri-dai-min}.  Also these cases are
interesting for synchronization phenomena, but we decided to not
include them in this paper since, as we will explain more deeply in
the sequel, they lead to quite different asymptotic results and so we
think that it is more appropriate to possibly deal with them
separately.  \\

\indent Finally, we mention that in literature we can find other works
concerning models of interacting urns, that consider interacting
mechanisms different from ours and are generally not focused on
synchronization. For instance, the model studied in \cite{mar-val}
describes a system of interacting units, modeled by P\'olya urns,
subject to perturbations and which occasionally break down. The
authors consider a system of interacting P\'olya urns arranged on a
$d$-dimensional lattice. Each urn contains initially $b$ black balls
and $1$ white ball. At each time step an urn is selected and a ball is
drawn from it: if the ball is white, a new white ball is added to the
urn; if it is black a ``fatal accident'' occurs and the urn becomes
unstable and it ``topples'' coming back to the initial configuration.
The toppling mechanism involves also the nearby urns. In
\cite{pag-sec} a class of discrete time stochastic processes generated
by interacting systems of reinforced urns is introduced and its
asymptotic properties analyzed. Given a countable set of urns, at each
time a ball is independently sampled from every urn in the system and
in each urn a random number of balls of the same color of the
extracted ball is added. The interaction arises since the number of
added balls depends also on the colors generated by the other urns as
well as on a common random factor.  In \cite{cir} the authors consider
a network of interacting urns displaced over a lattice. Every urn is
P\'olya-like and its reinforcement matrix is not only a function of
time (time contagion) but also of the behavior of the neighboring urns
(spatial contagion), and of a random component, which can represent
either simple fate or the impact of exogenous factors. In this way a
non-trivial dependence structure among the urns is built, and the
given construction is used to model different phenomena characterized
by cascading failures such as power grids and financial networks.  In
\cite{ben, che-luc, lima} a graph-based model, with urns at each
vertex and pair-wise interactions, is considered. Given a finite
connected graph, place a bin at each vertex. Two bins are called a
pair if they share an edge. At discrete times, a ball is added to each
pair of bins. In a pair of bins, one of the bins gets the ball with
probability proportional to its current number of balls raised by some
fixed power $\alpha>0$. The authors characterize the limiting behavior
of the proportion of balls in the bins for different values of the
parameter $\alpha$.  \\

\indent The rest of the paper is organized as follows. In
Section~\ref{section_model} we introduce the notation, we describe the
model and the leading
assumptions. Section~\ref{section_asymptotic_results} is concerned
with the main results established in the paper, while the relative
proofs are gathered in Section~\ref{section_proofs}.  Some meaningful
examples of reinforced random processes with a network-based
interaction are described in Section~\ref{section_examples}, in order
to apply the theoretical results provided in the paper to some
practical cases and to establish the corresponding asymptotic
behaviors.  In Section~\ref{section_applications} we illustrate some
statistical tools coming from the obtained theoretical results. In
particular, we propose an inferential procedure to test the structure
of the network which the interaction between the reinforced stochastic
processes is based on.  Finally, Section~\ref{section_variants} is
concerned with some possible variants of the model here presented. For
reader's convenience, the paper is also enriched by an exhaustive
appendix containing necessary definitions and technical results.

\section{The Model}\label{section_model}

Throughout the paper, we will adopt the following notation:
\begin{itemize}
\item[a)] Given a complex number $z$, ${\mathcal Re}(z)$ and
${\mathcal
    Im}(z)$ denote its real and imaginary parts, respectively,
  $\overline{z}$ denotes its conjugate and $|z|$ its modulus.
\item[b)] If $A$ is a matrix with complex entries, then $\overline{A}$
  denotes its conjugate, i.e. the matrix whose entries are the
  conjugates of the entries of $A$, and $A^{\top}$ indicates its
  transpose. Moreover, we denote by $|A|$ the sum of the modulus of
  its entries so that, if $A$ is equal to the row-column product of
  two matrices $B,\, C$, we have $|A|\leq |B|\, |C|$. Finally, $Sp(A)$
  indicates its spectrum, i.e. the set of all its eigenvalues repeated
  with multiplicity, and $\lambda_{\max}(A)$ indicates the sub-set of
  $Sp(A)$ containing the eigenvalues with maximum real part,
  i.e. $\lambda^*\in \lambda_{\max}(A)$ whenever ${\mathcal
    Re}(\lambda^*)=\max\{ {\mathcal Re}(\lambda):\, \lambda\in Sp(A)
  \}$. Moreover, we will denote by $I$ the identity matrix, whose
  dimension depends on the context.
\item[c)] A vector $\mathbf{v}$ is considered as a matrix with a
  single column, and hence all the notations stated in b) apply to
  $\mathbf{v}$. Moreover, $\|\mathbf{v}\|$ indicates the norm of the
  vector $\mathbf{v}$, i.e. $ \|\mathbf{v} \|^2 =
  \overline{\mathbf{v}}^{\top}\mathbf{v}$.  Finally, we will denote by
  $\mathbf{1}$ and by $\mathbf{0}$ the vectors whose entries are all
  ones and all zeros, respectively.
\end{itemize}

\noindent We now present the model.  Suppose to have a directed graph
$G=(V, E)$ with $V=\{1,...,N\}$ as the set of vertices and
$E\underline{\subset} V\times V$ as the set of edges.  Each edge
$(j,k)\in E$ represents the fact that the vertex $j$ has a direct
influence on the vertex $k$. We assume also to associate a weight
$w_{jk}\geq 0$ to each link in order to quantify how much $j$ can
influence $k$.  A weight equal to zero means that the edge is not
present. We set $W=[w_{j,k}]_{j,k\in V\times V}$ ({\em weighted
  adjacency matrix}) and we assume the weights to be normalized so
that $\sum_{j=1}^N w_{j,k}=1$. Finally, we suppose to have at each
vertex $j$ a reinforced stochastic process described by
$X^j=(X_{n,j})_{n\geq 0}$ such that, for each $n\geq 0$, the random
variables $\{X_{n+1,j}:\,j\in V\}$ are conditional independent given
${\mathcal F}_{n}$ with
\begin{equation}\label{interacting-1}
P(X_{n+1,j}=1\, |\, {\mathcal F}_n)=\sum_{k=1}^N w_{k,j} Z_{n,k}
\end{equation}
where, for each $k\in V$,
\begin{equation}\label{interacting-2}
Z_{n,k}=(1-r_{n-1})Z_{n-1,k}+r_{n-1}X_{n,k}
\end{equation}
with $0\leq r_{n-1}<1$ constants and ${\mathcal F}_n=\sigma(Z_{0,k}: k\in V)\vee
\sigma(X_{m,j}: j\in V,\, m\leq n )$. \\

\indent To express the above dynamics in a compact form, let us define
the vectors $\mathbf{X}_{n}=(X_{n,1},..,X_{n,N})^{\top}$ and
$\mathbf{Z}_{n}=(Z_{n,1},..,Z_{n,N})^{\top}$.  Hence, the dynamics can be
expressed as follows:
\begin{equation}\label{eq:dynamic-0}
E[\mathbf{X}_{n+1}|\mathcal{F}_{n}]=W^{\top}\,\mathbf{Z}_{n}
\end{equation}
where
\begin{equation}\label{eq:dynamic}
\mathbf{Z}_{n}\ =
\left(1-r_{n-1}\right)\mathbf{Z}_{n-1}\ +\ r_{n-1}\mathbf{X}_{n}.
\end{equation}
Moreover, the assumption about the normalization of the matrix $W$ can be
written as $W^{\top}\mathbf{1}=\mathbf{1}$.\\

\indent Throughout all the paper, we assume that the following additional
conditions hold.

\begin{ass}\label{ass:W_irreducible}
The weighted adjacency matrix $W$ is irreducible.
\end{ass}
This condition reflects a situation in which all the vertices are
connected among each others and hence there are no sub-systems with
independent dynamics (see \cite{ale-ghi} and
Subsection~\ref{subsection_reducible_W} for further details).

\begin{ass}\label{ass:r_n}
There exists a constant $c>0$ and $1/2<\gamma\leq 1$ such that
\begin{equation}\label{eq:condition_r_n_main}
\lim_{n\rightarrow\infty} n^{\gamma} r_n\ =\ c.
\end{equation}
\end{ass}
When $\gamma=1$, for a particular case covered by our analysis, we
will require a slightly stricter condition
than~\eqref{eq:condition_r_n_main}, that is:
\begin{equation}\label{eq:condition_r_n_bis_main}
nr_n - c\ =\ O\left(n^{-1}\right).
\end{equation}

This paper is concerned with the case $1/2 < \gamma \leq 1$, while the
case $\gamma \leq 1/2$ is not considered. Indeed, in
\cite{cri-dai-lou-min} it was established that, under soft assumptions
on the initial distribution, if the mean-field interaction is present
and $\sum_n r_n^2=+\infty$, then all the stochastic processes
$\{(Z_{n,j})_n:\, 1\leq j\leq N \}$ converge almost surely to the same
random variable $Z_\infty\in \{0,1\}$ a.s. Hence, although this case
is interesting for synchronization, we decided to focus here on the
case $1/2 <\gamma \leq 1$, for which soft assumptions on the initial
distribution lead to a limit random variable not concentrated only on
$\{0,1\}$.  \\

\indent Finally, we require the following condition:
\begin{ass}\label{ass:W_diagonalizable}
The weighted adjacency matrix $W$ is diagonalizable.
\end{ass}

This assumption implies that there exists a non-singular matrix
$\widetilde{U}$ such that
$\widetilde{U}^{\top}W(\widetilde{U}^{\top})^{-1}$ is diagonal with
elements $\lambda_j\in Sp(W)$.  Notice that each column $\mathbf{u}_j$
of $\widetilde{U}$ is a left eigenvector of $W$ associated to
$\lambda_j$.  Without loss of generality, we set $\|\mathbf{u}_j\|=1$.
Moreover, when the multiplicity of some $\lambda_j$ is bigger than
one, we set the corresponding eigenvectors to be orthogonal.  Then, if
we define $\widetilde{V}=(\widetilde{U}^{\top})^{-1}$, we have that
each column $\mathbf{v}_j$ of $\widetilde{V}$ is a right eigenvector
of $W$ associated to $\lambda_j$ such that
\begin{equation}\label{eq:relazioni-0}
\mathbf{u}_j^{\top}\,\mathbf{v}_j=1,\quad\mbox{ and }\qquad
\mathbf{u}_h^{\top}\,\mathbf{v}_j=0,\ \forall h\neq j.
\end{equation}

These constraints combined with the above assumptions on $W$
(precisely, $w_{j,k}\geq 0$, $W^{\top}\mathbf{1}=\mathbf{1}$ and Assumption
\ref{ass:W_irreducible}) imply, by Frobenius-Perron Theorem, that
$\lambda_1:=1$ is an eigenvalue of $W$ with multiplicity one,
$\lambda_{\max}(W)=\{1\}$ and
\begin{equation}\label{eq:relazioni-1}
\mathbf{u}_1=N^{-1/2}\mathbf{1}, \qquad
N^{-1/2}{\mathbf 1}^{\top}{\mathbf v}_1=1\qquad
\mbox{and}\qquad
[\mathbf{v}_1]_k:=v_{1,k}\in (0,+\infty)\; \forall k=1,\dots, N.
\end{equation}

Finally, throughout all the paper, we will use $U$ and $V$ to indicate
the sub-matrices of $\widetilde{U}$ and $\widetilde{V}$, respectively,
whose columns are the left and the right eigenvectors of
$W$ associated to $Sp(W)\setminus \{1\}$, that is
$\{\mathbf{u}_2,..,\mathbf{u}_N\}$ and
$\{\mathbf{v}_2,..,\mathbf{v}_N\}$, respectively, and we will denote
by $\lambda^*$ an eigenvalue belonging to $Sp(W)\setminus\{1\}$ such
that
$$ {\mathcal Re}(\lambda^*)=\max\left\{ {\mathcal Re}(\lambda_j):\,
\lambda_j\in Sp(W)\setminus\{1\}\right\}.$$

\section{Main results}\label{section_asymptotic_results}

In this section, we present our main results, which regard the
asymptotic behavior of the process $\mathbf{Z}_{n}$.  We refer to the
appendix for a brief review of the notion of stable convergence.
\\

\indent Let us recall the assumptions stated in Section
\ref{section_model}.  We start by providing a first-order asymptotic
result concerning the almost sure convergence of $\mathbf{Z}_{n}$.

\begin{theo}\label{th:sincro} (Synchronization)\\
There exists a random variable $Z_{\infty}$ with values in $[0,1]$ such that
\begin{equation}\label{eq:as_convergence_Z}
\mathbf{Z}_{n}\ \stackrel{a.s.}{\longrightarrow}\ Z_{\infty}\mathbf{1}.
\end{equation}
\end{theo}

This result states that the stochastic processes $\{(Z_{n,j})_n:\,
1\leq j\leq N\}$ located at the different vertices synchronize,
i.e. all of them converge almost surely toward the same random
variable $Z_\infty$.  It is interesting to note that this result holds
true without any assumption on the initial configuration
$\mathbf{Z}_{0}$ and for any choice of the weighted adjacency matrix
$W$ with the required assumptions.  \\

\indent We now focus on the second-order asymptotic results concerning
the process $(\mathbf{Z}_{n})_n$.  First, we present a central limit
theorem in the sense of stable convergence, that establishes the rate
of convergence to the limit $Z_{\infty}\mathbf{1}$ determined in
Theorem~\ref{th:sincro} and the relative asymptotic random variance.

\begin{theo}\label{thm:asymptotics_Z}(CLT for convergence)\\
The following hold:
\begin{itemize}
\item[(a)] For $1/2<\gamma<1$, then
\begin{equation*}
n^{\gamma-\frac{1}{2}}
\left(\mathbf{Z}_{n}-Z_{\infty}\mathbf{1}\right)\
\stackrel{d}{\longrightarrow}\
\mathcal{N}
\left(\ 0\ ,\ Z_{\infty}(1-Z_{\infty})\widetilde{\Sigma}_{\gamma}\ \right),
\ \ \ \ stably
\end{equation*}
where
\begin{equation}\label{def:sigmatilde-gamma}
\widetilde{\Sigma}_{\gamma}:=
\widetilde{\sigma}_{\gamma}^2{\mathbf 1}{\mathbf 1}^{\top}
\qquad\mbox{ and }\qquad
\widetilde{\sigma}_{\gamma}^2:=
\frac{c^2\,\|\mathbf{v}_1\|^2}{N(2\gamma-1)}
>0.
\end{equation}

\item[(b)] For $\gamma=1$, if ${\mathcal Re}(\lambda^{*})<1-(2c)^{-1}$, then
\begin{equation*}\label{eq:CLT_Z_less}
\sqrt{n}\left(\mathbf{Z}_{n}-Z_{\infty}\mathbf{1}\right)\
\stackrel{d}{\longrightarrow}\
\mathcal{N}
\left(\ 0\ ,\ Z_{\infty}(1-Z_{\infty})(\widetilde{\Sigma}_1+\widehat{\Sigma}_1)\
\right),\ \ \ \
stably
\end{equation*}
where $\widetilde{\Sigma}_1$ is defined as in~\eqref{def:sigmatilde-gamma}
with $\gamma=1$,
\begin{equation}\label{def:sigmahat1}
\widehat{\Sigma}_1:=U\widehat{S}_1U^{\top}\qquad\mbox{ and }\qquad
[\widehat{S}_1]_{h,j}:=
\frac{c^2}{2c-c(\lambda_h+\lambda_j)-1}(\mathbf{v}_{h}^{\top}\mathbf{v}_{j}),\
\mbox{ with }\ 2\leq h,j\leq N.
\end{equation}

\item[(c)] For $\gamma=1$, if ${\mathcal Re}(\lambda^{*})=1-(2c)^{-1}$
  and~\eqref{eq:condition_r_n_bis_main} holds, then
\begin{equation*}\label{eq:CLT_Z_less}
\frac{\sqrt{n}}{\sqrt{\ln(n)}}\left(\mathbf{Z}_{n}-Z_{\infty}\mathbf{1}\right)\
{\longrightarrow}\
\mathcal{N}\left(\ 0\ ,\ Z_{\infty}(1-Z_{\infty})\widehat{\Sigma}_1^*\ \right),
\ \ \ \ stably
\end{equation*}
where
\begin{equation}\label{def:sigmahat1_star}
\widehat{\Sigma}_1^*:=U\widehat{S}_1^*U^{\top}\qquad\mbox{ and }\qquad
[\widehat{S}_1^*]_{h,j}:=\begin{cases}
c^2 (\mathbf{v}_{h}^{\top}\mathbf{v}_{j})
\quad &\mbox{if } \lambda_h+\lambda_j = 2-c^{-1},
\\
0\quad &\mbox{if } \lambda_h+\lambda_j\neq 2-c^{-1},
\end{cases}
\
\mbox{ with }\ 2\leq h,j\leq N.
\end{equation}
\end{itemize}
\end{theo}

Notice that the matrix $\widehat{S}_1^*$ defined
in~\eqref{def:sigmahat1_star} can never be null, as stated more ahead
in Theorem~\ref{thm:real_positive_variances}.

\begin{rem}
\rm Notice that $\widetilde{\sigma}_{\gamma}^2$ is decreasing with the
size $N$ of the network and so, for cases (a) and (b), the larger the
size of the network, the lower the asymptotic variance.  Moreover,
fixed $N$ and $\gamma$, since by~\eqref{eq:relazioni-0}
and~\eqref{eq:relazioni-1} we have
$\|\mathbf{v}_1\|^2=\|\mathbf{u}_1+(\mathbf{v}_1-\mathbf{u}_1)\|^2=
1+\|\mathbf{v}_1-\mathbf{u}_1\|^2\geq 1$ and $\|\mathbf{v}_1\|^2\leq
N$, we can obtain the following lower and upper bounds for
$\widetilde{\sigma}_{\gamma}^2$ (not depending on $W$):
$$\frac{c^2}{N(2\gamma-1)}\
\leq\ \widetilde{\sigma}_{\gamma}^2\
\leq\ \frac{c^2}{(2\gamma-1)},$$
where the lower bound is achieved when $\mathbf{v}_1=\mathbf{u}_1$,
i.e.  when $W$ is doubly stochastic.
\end{rem}

Given the long-run synchronization stated in Theorem~\ref{th:sincro},
it is interesting to establish the rate of
synchronization, that is the convergence rate of the difference
$(Z_{n,j}-Z_{n,k})_n$ to zero for $j\neq k$ and to characterize the
relative asymptotic distribution.  The following result achieves this
goal.

\begin{theo}\label{thm:asymptotics_Z_j_Z_k} (CLT for synchronization)\\
For any $j,k\in\{1,..,N\}$, $j\neq k$, we have:
\begin{itemize}
\item[(a)] For $1/2<\gamma<1$, then
\begin{equation*}\label{eq:CLT_Z_j_Z_k_gamma}
n^{\frac{\gamma}{2}}\left(Z_{n,j}-Z_{n,k}\right)\ {\longrightarrow}\
\mathcal{N}\left(\ 0\ ,\ Z_{\infty}(1-Z_{\infty})\Sigma_{\gamma, j,k}\ \right),
\ \ \ \ stably
\end{equation*}
where $\Sigma_{\gamma,j,k}:=
[\widehat{\Sigma}_\gamma]_{j,j}+
[\widehat{\Sigma}_\gamma]_{k,k}-2[\widehat{\Sigma}_\gamma]_{j,k}$,
\begin{equation}\label{def:sigmahat_gamma}
\widehat{\Sigma}_{\gamma}:=U\widehat{S}_{\gamma}U^{\top}\qquad\mbox{ and }\qquad
[\widehat{S}_{\gamma}]_{h,j}:=
\frac{c}{2-(\lambda_h+\lambda_j)} (\mathbf{v}_{h}^{\top}\mathbf{v}_{j}),\
\mbox{ with }\ 2\leq h,j\leq N.
\end{equation}

\item[(b)] For $\gamma=1$, if ${\mathcal Re}(\lambda^{*})<1-(2c)^{-1}$, then
\begin{equation*}\label{eq:CLT_Z_j_Z_k_less}
\sqrt{n}\left(Z_{n,j}-Z_{n,k}\right)\ {\longrightarrow}\
\mathcal{N}\left(\ 0\ ,\ Z_{\infty}(1-Z_{\infty})\Sigma_{1,j,k}\ \right),
\ \ \ \ stably
\end{equation*}
where
$\Sigma_{1,j,k}:=
[\widehat{\Sigma}_1]_{j,j}+[\widehat{\Sigma}_1]_{k,k}-2[\widehat{\Sigma}_1]_{j,k}$
and
$\widehat{\Sigma}_1$ is defined in \eqref{def:sigmahat1}.

\item[(c)] For $\gamma=1$, if ${\mathcal Re}(\lambda^{*})=1-(2c)^{-1}$
  and~\eqref{eq:condition_r_n_bis_main} holds, then
\begin{equation*}\label{eq:CLT_Z_j_Z_k_less}
\frac{\sqrt{n}}{\sqrt{\ln(n)}}
\left(Z_{n,j}-Z_{n,k}\right)\ {\longrightarrow}\
\mathcal{N}\left(\ 0\ ,\ Z_{\infty}(1-Z_{\infty})\Sigma_{1,j,k}^*\ \right),
\ \ \ \ stably
\end{equation*}
where
$\Sigma_{1,j,k}^*:=
[\widehat{\Sigma}_1^*]_{j,j}+
[\widehat{\Sigma}_1^*]_{k,k}-2[\widehat{\Sigma}_1^*]_{j,k}$
and $\widehat{\Sigma}_1^*$ is defined in \eqref{def:sigmahat1_star}.
\end{itemize}
\end{theo}

\begin{rem}\label{rem:Z_hat}
\rm In the particular case when $W$ is symmetric, the eigenvectors of
$W$ are real, $U=V$ and $V^{\top}V=I$.  As a consequence, the matrices
$\widehat{S}_{\gamma}$, $\widehat{S}_1$ and $\widehat{S}_1^*$ are
diagonal, with elements $c[2(1-\lambda_j)]^{-1}$,
$c[2(1-\lambda_j)-c^{-1}]^{-1}$ and $c^2
\mathbbm{1}_{\{\lambda_j=1-(2c)^{-1}\}}$, respectively, where
$\lambda_j\in Sp(W)\setminus\{1\}$. Moreover, in this case, we have
$\mathbf{u}_1=\mathbf{v}_1=N^{-1/2}\mathbf{1}$ and
$UU^{\top}=UV^{\top}=(I-N^{-1}\mathbf{1}\mathbf{1}^{\top})$ (see Subsection
\ref{subsection_preliminaries} for details). Notice that, for
instance, this is the case of undirected graphs.
\end{rem}

In order to ensure that
  Theorem~\ref{thm:asymptotics_Z} and~\ref{thm:asymptotics_Z_j_Z_k}
  provide the right convergence rates of $(Z_{n,j})_n$ to $Z_{\infty}$ and
  of $(Z_{n,j}-Z_{n,k})_n$ to zero, respectively, we need to have
  $[\widehat{\Sigma}_{1}]_{j,j}\geq 0$, $[\widehat{\Sigma}^*_{1}]_{j,j}>0$,
$\Sigma_{\gamma,j,k}>0$,
  $\Sigma_{1,j,k}>0$, $\Sigma_{1,j,k}^*>0$ and
\begin{equation}\label{Z01}
P(Z_{\infty}=0)\ +\ P(Z_{\infty}=1)\ <\ 1.
\end{equation}

\indent The result below deals with the first set of conditions.

\begin{theo}\label{thm:real_positive_variances}
We have:
\begin{itemize}
\item[(a)] For $1/2<\gamma<1$, $\widehat{\Sigma}_{\gamma}$ is a
  positive semi-definite real matrix of rank $(N-1)$ and
  $\mathbf{v}_1^{\top}\widehat{\Sigma}_{\gamma}\mathbf{v}_1=0$; in addition,
  $\Sigma_{\gamma,j,k}>0$ for any $1\leq j\neq k\leq N$.

\item[(b)] For $\gamma=1$, if ${\mathcal
  Re}(\lambda^{*})<1-(2c)^{-1}$, then $\widehat{\Sigma}_1$ is a
  positive semi-definite real matrix of rank $(N-1)$ and
  $\mathbf{v}_1^{\top}\widehat{\Sigma}_{1}\mathbf{v}_1=0$; in addition,
  $\Sigma_{1,j,k}>0$ for any $1\leq j\neq k\leq N$.

\item[(c)] For $\gamma=1$, if ${\mathcal Re}(\lambda^{*})=1-(2c)^{-1}$, define
\begin{equation}\label{def:A_star}
A^{*}\ :=\
\left\{\ \lambda_j\in Sp(W),\ {\mathcal Re}(\lambda_j)=1-(2c)^{-1}\ \right\}
\end{equation}
and let $m^*$ be the cardinality of $A^{*}$; then
$\widehat{\Sigma}_1^*$ is a positive semi-definite real matrix of
rank $m^*$ and $\mathbf{v}_j^{\top}\widehat{\Sigma}_{1}^*\mathbf{v}_j=0$ for
any $j$ such that $\lambda_j\notin A^{*}$; moreover,
$[\widehat{\Sigma}_1^*]_{jj}>0$ when $u_{h,j}\neq 0$ for some $h$ such
that $\lambda_h\in A^{*}$ and
$\Sigma_{1,j,k}>0$ when $u_{h,j}\neq u_{h,k}$ for some $h$ such
that $\lambda_h\in A^{*}$.
\end{itemize}
\end{theo}

Finally, we give two results concerning the distribution of $Z_\infty$,
of which the last one deals with condition \eqref{Z01}.

\begin{theo}\label{thm:Z-no-atomi}
We have $P(Z_{\infty}=z)=0$ for any $z\in (0,1)$.
\end{theo}

\begin{theo}\label{thm:distribution_Z}
If we have
\begin{equation}\label{Z01_Hp}
P\left(\ \bigcap_{k=1}^N\{Z_{0,k}=0\}\ \right)\
+\ P\left(\ \bigcap_{k=1}^N\{Z_{0,k}=1\}\ \right)\ <\ 1,
\end{equation}
then condition \eqref{Z01} is verified.
\end{theo}

\begin{rem}
\rm In case (a), i.e. $1/2<\gamma<1$, since $\gamma/2>\gamma-1/2$ we
have that the rate at which two stochastic processes
$(Z_{n,j})_n,\,(Z_{n,k})_n$ positioned in any pair of different
vertices $(j,k)$ of the network synchronize is greater than the rate
at which they converge to $Z_{\infty}$, i.e. synchronization of the
stochastic processes at the vertices is faster then their convergence
to the limit random variable.
\end{rem}

\section{Proofs}\label{section_proofs}

This section contains all the proofs of the results presented in the
previous Section~\ref{section_asymptotic_results}.

\subsection{Preliminary relations and basic idea}
\label{subsection_preliminaries}

We start by recalling that, given the eigen-structure of $W$ described
in Section~\ref{section_model}, the matrix ${\mathbf
  u}_1{\mathbf v}_1^{\top}$ has real entries and
the following relations hold:
\begin{equation}\label{eq:relazioni-2}
V^{\top}\,\mathbf{u}_1=U^{\top}\,\mathbf{v}_1=\mathbf{0},\quad
V^{\top}\,U=U^{\top}\,V=I\quad\mbox{and}\quad
I={\mathbf u}_1{\mathbf v}_1^{\top} + UV^{\top},
\end{equation}
which implies that the matrix $UV^{\top}$ has real entries (Notice
that in~\eqref{eq:relazioni-2} the identity matrices have different
dimensions). Moreover, denoting by $D$ the diagonal matrix whose
elements are $\lambda_j\in Sp(W)\setminus\{1\}$, we can decompose the
matrix $W^{\top}$ as follows:
\begin{equation}\label{eq:decomposition-matrix}
W^{\top}\ =\ {\mathbf u}_1{\mathbf v}_1^{\top}\ +\ UDV^{\top}.
\end{equation}

\indent Now, in order to understand the asymptotic behavior of the
stochastic process $(\mathbf{Z}_{n})_n$, let us express the
dynamics~\eqref{eq:dynamic} as follows:
\begin{equation}\label{eq:dynamic_SA}
\mathbf{Z}_{n+1}-\mathbf{Z}_{n}\ =\ -r_n\left(I-W^{\top}\right)\mathbf{Z}_{n}\
+\ r_n\Delta\mathbf{M}_{n+1}.
\end{equation}
where $\Delta\mathbf{M}_{n+1}=(\mathbf{X}_{n+1}-W^{\top}\mathbf{Z}_{n})$ is a
martingale increment with respect to $({\mathcal F}_{n})_n$. It follows:
\begin{itemize}
\item[(a)] since ${\mathbf v}_1^{\top}W^{\top}=(W{\mathbf
    v}_1)^{\top}={\mathbf v}_1^{\top}$, we have
${\mathbf v}_1^{\top} (I-W^{\top})=\mathbf{0}$
  and so, from \eqref{eq:dynamic_SA}, we deduce that the stochastic
  process $({\mathbf v}_1^{\top} \mathbf{Z}_{n})_n$ is a bounded real
  martingale;
\item[(b)] by \eqref{eq:relazioni-2}, we have
  $\mathbf{Z}_{n}-{\mathbf u}_1 ({\mathbf v}_1^{\top}
  \mathbf{Z}_{n})=UV^{\top}\mathbf{Z}_n$ and so the dynamics of this
  multi-dimensional real stochastic process can be easily obtained from
  \eqref{eq:dynamic_SA}.
\end{itemize}
Hence, the basic idea is to decompose $\mathbf{Z}_{n}$ into two terms,
establish the corresponding asymptotic results for each term
separately and then combine them together to characterize the
asymptotic behavior of $\mathbf{Z}_{n}$.  More precisely, the process
$\mathbf{Z}_{n}$ can be decomposed as follows:
\begin{equation}\label{eq:decomposition_Z}
\mathbf{Z}_{n}=\widetilde{Z}_n\mathbf{1} + \widehat{\mathbf{Z}}_{n}
=\mathbf{u}_1\sqrt{N}\widetilde{Z}_n + \widehat{\mathbf{Z}}_{n},
\;\mbox{where}\;
\left\{
\begin{aligned}
&\widetilde{Z}_{n}=N^{-1/2}\,\mathbf{v}_1^{\top}\,\mathbf{Z}_{n}, \\
&\widehat{\mathbf{Z}}_{n}={\mathbf Z}_n-{\mathbf 1}{\widetilde Z}_n
=(I-\mathbf{u}_1\mathbf{v}_1^{\top})\mathbf{Z}_n
=U\,V^{\top}\,\mathbf{Z}_{n}.
\end{aligned}
\right.
\end{equation}
Then, the asymptotic behavior of the stochastic process
$(\mathbf{Z}_{n})_n$ is obtained by establishing the asymptotic
behavior of $(\widetilde{Z}_{n})_n$ and
$(\widehat{\mathbf{Z}}_{n})_n$.

\begin{rem}\label{rem:Z_tilde_doubly_stoch}
\rm In the particular case of $W$ doubly stochastic, we have
$\mathbf{v}_1=\mathbf{u}_1=N^{-1/2}\mathbf{1}$.
As a consequence, we have
\begin{equation}\label{eq:doubly_stoch_tilde_medione}
\widetilde{Z}_{n}=N^{-1}\mathbf{1}^{\top}\mathbf{Z}_{n}
=N^{-1}\sum_{j=1}^N Z_{n,j},
\end{equation}
which represents the average of the stochastic processes in the
network, and $\widehat{\mathbf{Z}}_{n}=
\left(I-N^{-1}\mathbf{1}\mathbf{1}^{\top}\right)\mathbf{Z}_n$.  Notice that
the assumed normalization $W^{\top}\mathbf{1}=\mathbf{1}$ implies that
symmetric matrices $W$ are also doubly stochastic. Therefore, the
above equalities hold for any undirected graph for which $W$ is
obviously symmetric by definition.
\end{rem}

\subsection{Proof of Theorem \ref{th:sincro} (Synchronization)}
\label{subsection-synchro}

By decomposition~\eqref{eq:decomposition_Z}, i.e.
$$\mathbf{Z}_{n} =
\widetilde{Z}_n\mathbf{1}+\widehat{\mathbf{Z}}_{n},$$ the proof of
Theorem \ref{th:sincro} follows by establishing the following two
results:
\begin{itemize}
\item[(i)] ${\widetilde Z}_n\stackrel{a.s.}\longrightarrow Z_{\infty}$,
\item[(ii)] $\widehat{\mathbf{Z}}_n\stackrel{a.s.}\longrightarrow 0$.
\end{itemize}

\indent Concerning part (i), let us consider the real-valued
stochastic process $(\widetilde{Z}_{n})_n$ defined for any $n\geq 0$
as ${\widetilde Z}_n= N^{-1/2}{\mathbf v}_1^{\top}{\mathbf Z}_n$.  Since all
the elements of ${\mathbf v}_1$ are positive and since
\eqref{eq:relazioni-1} holds, the elements of $N^{-1/2}{\mathbf v}_1$
can be seen as the weights of a convex combination and hence
$\min_j\{Z_{n,j}\}\leq{\widetilde Z}_n\leq\max_j\{Z_{n,j}\}$ for any
$n$, which implies $0\leq{\widetilde Z}_n\leq 1$.  Moreover, it is
easy to see that $(\widetilde{Z}_{n})_n$ is an $\mathcal F$-martingale, since
from~\eqref{eq:dynamic_SA} its dynamics can be expressed as follows:
\begin{equation}\label{eq:dynamic_SA_tilde}
\widetilde{Z}_{n+1}-\widetilde{Z}_{n}\ =\
N^{-1/2} r_n\left(\mathbf{v}_1^{\top}\Delta\mathbf{M}_{n+1}\right).
\end{equation}
Hence, we immediately get
\begin{equation}\label{eq:conv_Z_tilde}
{\widetilde Z}_n\stackrel{a.s.}\longrightarrow Z_{\infty},
\end{equation}
where $Z_\infty$ is a random variable with values in $[0,1]$.
This concludes the proof of part (i).\\

\indent Concerning part (ii), let us consider the multi-dimensional
stochastic process $(\widehat{\mathbf{Z}}_{n})_n$, with real entries,
defined in~\eqref{eq:decomposition_Z}.  In order to find the dynamics
of this process, we firstly observe that, by decomposition
\eqref{eq:decomposition_Z} and the fact that $W^{\top}{\mathbf u}_1=
({\mathbf u}_1^{\top}W)^{\top}={\mathbf u}_1$, we have
  $$
\left(I-W^{\top}\right)\mathbf{Z}_{n}=
\left(I-W^{\top}\right)
\left(\mathbf{u}_1\sqrt{N}\widetilde{Z}_n+\widehat{\mathbf{Z}}_{n}\right)
=
\left(I-W^{\top}\right)\widehat{\mathbf{Z}}_{n}
$$
and so the dynamics~\eqref{eq:dynamic_SA} of $\mathbf{Z}_{n}$ can be
  rewritten as
\begin{equation}\label{eq:dynamic_SA-2}
\mathbf{Z}_{n+1}-\mathbf{Z}_{n}\ =\
-r_n\left(I-W^{\top}\right)\widehat{\mathbf{Z}}_{n}\ +\
r_n\Delta\mathbf{M}_{n+1}.
\end{equation}
Then, if we multiply the dynamics~\eqref{eq:dynamic_SA-2} by $UV^{\top}$ and
use decomposition \eqref{eq:decomposition-matrix} and the
relations \eqref{eq:relazioni-2}, we obtain
\begin{equation}\label{eq:dynamic_SA_hat}
\begin{split}
\widehat{\mathbf{Z}}_{n+1}-\widehat{\mathbf{Z}}_{n}
&=
-r_n
\left[UV^{\top}-UV^{\top}(\mathbf{u}_1\mathbf{v}_1^{\top}+ UDV^{\top})\right]
\widehat{\mathbf{Z}}_{n}
+r_nUV^{\top}\,\Delta \mathbf{M}_{n+1}\\
&=
-r_n
(UV^{\top}-UDV^{\top})\widehat{\mathbf{Z}}_{n}
+r_nUV^{\top}\,\Delta \mathbf{M}_{n+1}\\
&=
-r_nU(I-D)V^{\top}\widehat{\mathbf{Z}}_{n}+
r_n UV^{\top}\,\Delta\mathbf{M}_{n+1},
\end{split}
\end{equation}
where $I$ in~\eqref{eq:dynamic_SA_hat} is a
$(N-1)\times(N-1)$-identity matrix.
We are now ready for proving that
this multi-dimensional stochastic process converges a.s. to
$\mathbf{0}$.

\begin{theo}\label{th:as_conv_Z_hat}
We have
\begin{equation}\label{eq:as_convergence_Z_hat}
\widehat{\mathbf{Z}}_{n}\ \stackrel{a.s.}{\longrightarrow}\ \mathbf{0}.
\end{equation}
\end{theo}

\begin{proof} Let us consider the $(N-1)$-dimensional complex
  random vector defined as
  $\mathbf{Z}_{V,n}=V^{\top}\widehat{\mathbf{Z}}_{n}$. Since we have
  $\widehat{\mathbf{Z}}_{n}=U\mathbf{Z}_{V,n}$ by
  \eqref{eq:relazioni-2}, it is enough to prove that
  $\mathbf{Z}_{V,n}$ converges almost surely to $\mathbf{0}$. To this
  purpose, we observe that the dynamics of $\mathbf{Z}_{V,n}$ can be
  obtained from \eqref{eq:dynamic_SA_hat} multiplying by $V^{\top}$:
$$\mathbf{Z}_{V,n+1}\ =\ (I-r_n(I-D))\mathbf{Z}_{V,n}\
+\ r_nV^{\top}\Delta \mathbf{M}_{n+1},
$$
where $I$ here indicates a $(N-1)\times(N-1)$-identity matrix.
Hence, recalling that $E[\Delta\mathbf{M}_{n+1}\,|\,{\mathcal
    F}_{n}]=0$, we obtain
\begin{equation*}
\begin{split}
E\left[\|\mathbf{Z}_{V,n+1}\|^2|\mathcal{F}_n\right]
&=
E\left[
\overline{\mathbf{Z}}_{V,n+1}^{\top}\, \mathbf{Z}_{V,n+1}\,|\, {\mathcal F}_{n}
\right]
\\
&=
\overline{\mathbf{Z}}_{V,n}^{\top}(I-r_n(I-\overline{D}))
(I-r_n(I-D))\mathbf{Z}_{V,n}\
+\ r_n^2
E\left[\Delta\mathbf{M}_{n+1}^{\top}\overline{V}V^{\top}\Delta \mathbf{M}_{n+1}
\,|\,\mathcal{F}_n\right]
\\
&=
\overline{\mathbf{Z}}_{V,n}^{\top}\, \mathbf{Z}_{V,n}
-r_n\overline{\mathbf{Z}}_{V,n}^{\top}
\left(2I-\overline{D}-D\right)\mathbf{Z}_{V,n}
+r_n^2\xi_{n},
\end{split}
\end{equation*}
where $(\xi_{n})_n$ is a suitable bounded sequence of ${\mathcal
  F}_{n}$-measurable random variables. Since
$\mathcal{R}e(\lambda_j)<1$ for any $\lambda_j\in
Sp(W)\setminus\{1\}$, the matrix $2I-(\overline{D}+D)$ is positive
definite and hence we can write
\begin{equation*}
E\left[
\|{\mathbf{Z}}_{V,n+1}\|^2\,|\, {\mathcal F}_{n}
\right]
\leq
\|{\mathbf{Z}}_{V,n}\|^2\ +\  O(r_n^2).
\end{equation*}
Since $\sum_n r_n^2<+\infty$ for $1/2<\gamma\leq 1$, we can conclude
that the real stochastic process $(\|{\mathbf{Z}}_{V,n}\|^2)_n$ is a
positive almost supermartingale and so (see \cite{rob-sie}) it
converges almost surely (and in mean since it is also bounded). In
order to prove that the limit is zero, it is enough to prove that
$E[\|\mathbf{Z}_{V,n}\|^2]$ converges to zero. To this end, we observe
that, from the above computations, we obtain
\begin{equation*}
\begin{split}
E[\|\mathbf{Z}_{V,n+1}\|^2]\ &=\
E[\overline{\mathbf{Z}}_{V,n}^{\top}(I-r_n(I-\overline{D}))(I-r_n(I-D))
\mathbf{Z}_{V,n}]\
+\ r_n^2E[\Delta\mathbf{M}_{n+1}^{\top}\overline{V}V^{\top}\Delta \mathbf{M}_{n+1}]
\\
&\leq
E[\overline{\mathbf{Z}}_{V,n}^{\top}(I-r_n(I-\overline{D}))(I-r_n(I-D))
\mathbf{Z}_{V,n}]\
+\ C_1 r_n^2
\end{split}
\end{equation*}
for a suitable constant $C_1\geq 0$. Then, we note that the elements
of the diagonal matrix above can be written as follows
$$[(I-r_n(I-\overline{D}))(I-r_n(I-D))]_{jj}\ =\ 1-2r_n(1-{\mathcal
  R}e(\lambda_j))+r^2_n\, |1-\lambda_j|^2.$$ Hence, setting
$a_j=1-{\mathcal R}e(\lambda_j)$ and $a^*=\min_j\{a_j\}=1-{\mathcal
  R}e(\lambda^*)$ (we recall that $\lambda^{*}$ indicates an
eigenvalue belonging to $\lambda_{max}(D)$), we have that
\[\begin{aligned}
E[\overline{\mathbf{Z}}_{V,n}^{\top}
(I-r_n(I-\overline{D}))(I-r_n(I-D))\mathbf{Z}_{V,n}]\ &&\leq&\
\sum_{j=2}^{N}(1-2a_j r_n)E[\overline{Z}^j_{V,n}Z^j_{V,n}]\ +C_2 r^2_n\\
&&\leq
&\ (1-2a^* r_n)E[\|\mathbf{Z}_{V,n}\|^2]\ +C_2 r^2_n
\end{aligned}\]
for a suitable constant $C_2\geq 0$. Then, setting
$x_n:=E[\|\mathbf{Z}_{V,n}\|^2]$, we can write
$$x_{n+1}\ \leq\ (1-2a^*r_n)x_n\ +\ (C_1+C_2)r_n^2.$$ Since ${\mathcal
  R}e(\lambda^*)<1$, we have $a^*>0$, which implies $\lim_n x_n=0$
(see \cite{cri-dai-lou-min}). The proof is thus concluded.
\end{proof}

\indent Note that, by the synchronization result given in Theorem
\ref{th:sincro}, we can state that
\begin{equation}\label{eq:multidim-limite-conditional-M}
E[(\Delta{\mathbf M}_{n+1})(\Delta{\mathbf M}_{n+1})^{\top}
\,|\,{\mathcal F}_n]
\stackrel{a.s.}\longrightarrow Z_{\infty}(1-Z_{\infty})I.
\end{equation}
Indeed, since $\{X_{n+1,k}:\, j=1,\dots, N\}$ are conditionally independent
given $\mathcal{F}_n$, we have
\begin{equation}\label{eq:covar-conditional-M}
E[\Delta{M}_{n+1,h}\Delta{M}_{n+1,k}\,|\,{\mathcal F}_n]=0\quad
\mbox{ for } h\neq k;
\end{equation}
while, for each $k$, we have
\begin{equation}\label{eq:var-conditional-M}
E[(\Delta{ M}_{n+1,k})^2\,|\,{\mathcal F}_n]=
\left(\sum_{j=1}^N w_{j,k}Z_{n,j}\right)
\left(1-\sum_{j=1}^N w_{j,k}Z_{n,j}\right).
\end{equation}
From this last equality, using synchronization and the normalization
$W^{\top}{\mathbf 1}={\mathbf 1}$, we immediately obtain
\begin{equation}\label{eq:limite-conditional-M}
E[(\Delta{ M}_{n+1,k})^2\,|\,{\mathcal F}_n]
\stackrel{a.s.}\longrightarrow Z_{\infty}(1-Z_{\infty}).
\end{equation}

\subsection{A CLT for $\widetilde{Z}_{n}$}
\label{subsection_asymptotic_results_Z_tilde}

The following result gives a central limit theorem for the real-valued
stochastic process $(\widetilde{Z}_{n})_n$.

\begin{theo}\label{thm:asymptotics_Z_tilde}
For $1/2<\gamma\leq 1$, we have
\begin{equation}\label{eq:CLT_Z_tilde}
n^{\gamma-\frac{1}{2}}
\left(\widetilde{Z}_{n}-Z_{\infty}\right)\
{\longrightarrow}\
\mathcal{N}
\left(\ 0\ ,\ \widetilde{\sigma}_{\gamma}^2\,Z_{\infty}(1-Z_{\infty})\ \right)\
\ \ \ stably,
\end{equation}
where $\widetilde{\sigma}_{\gamma}^2$ is defined
in~\eqref{def:sigmatilde-gamma}.  The above convergence is also in the
sense of the almost sure conditional convergence w.r.t. ${\mathcal
  F}=({\mathcal F}_n)_n$.
\end{theo}

\begin{proof} We want to
apply Theorem \ref{fam_tri_vet_as_inf}.  Let us consider, for each
$n\geq 1$ the filtration $({\mathcal F}_{n,h})_h$ and the process
$(L_{n,h})_h$ defined by
\begin{equation*}
{\mathcal F}_{n,0}={\mathcal F}_{n,1}={\mathcal F}_n, \qquad
L_{n,0}=L_{n,1}=0
\end{equation*}
and, for $h\geq 2$,
\begin{equation*}
{\mathcal F}_{n,h}={\mathcal F}_{n+h-1}, \qquad
L_{n,h}=n^{\gamma-\frac{1}{2}}({\widetilde Z}_n-{\widetilde Z}_{n+h-1}).
\end{equation*}
By \eqref{eq:dynamic_SA_tilde} and \eqref{eq:conv_Z_tilde}, the
process $(L_{n,h})_h$ is a martingale w.r.t. $({\mathcal F}_{n,h})_h$
which converges (for $h\to +\infty$) a.s. and in $L^1$ to the random
variable $L_{n,\infty}=n^{\gamma -\frac{1}{2}}(Z_n-Z_{\infty})$. In
addition, the increment $Y_{n,j}=L_{n,j}-L_{n,j-1}$ is equal to zero
for $j=1$ and, for $j\geq 2$, it coincides with a random variable of
the form $n^{\gamma
  -\frac{1}{2}}(\widetilde{Z}_k-\widetilde{Z}_{k+1})$ with $k\geq
n$. Therefore, again by \eqref{eq:dynamic_SA_tilde}, we have
\[\begin{aligned}
\sum_{j\geq1} Y_{n,j}^2\ &&=&\
n^{2\gamma-1}\sum_{k\geq n}(\widetilde{Z}_{k}-\widetilde{Z}_{k+1})^2\ =\
N^{-1} n^{2\gamma-1}\sum_{k\geq n} r_k^2 (\mathbf{v}_1^{\top}\Delta\mathbf{M}_{n+1})^2\\
&&\stackrel{a.s.}{\sim}&\
N^{-1}c^2n^{2\gamma-1}\sum_{k\geq n}k^{-2\gamma}(\mathbf{v}_1^{\top}
\Delta\mathbf{M}_{n})^2\\
&&\stackrel{a.s.}{\longrightarrow}
&\ \frac{c^2}{N}\frac{\|\mathbf{v}_1\|^2}{(2\gamma-1)}Z_{\infty}(1-Z_{\infty}),
\end{aligned}\]
where the last part follows by applying \cite[Lemma 4.1]{cri-dai-min}
and by noticing that \eqref{eq:limite-conditional-M} implies
$$
E[(\mathbf{v}_1^{\top}\Delta\mathbf{M}_{n+1})^2|\mathcal{F}_{n}]\ =\
\sum_{k=1}^N v^2_{1,k}\,E[(\Delta M_{n+1,k})^2|\mathcal{F}_{n}]\
\stackrel{a.s.}{\longrightarrow}\
\sum_{j=1}^N v^2_{1,k}Z_{\infty}(1-Z_{\infty})\ =\
\|\mathbf{v}_1\|^2 Z_{\infty}(1-Z_{\infty}).
$$
Finally, again by \eqref{eq:dynamic_SA_tilde}, we have
$$
Y_n^*=\sup_{j\geq 1}|Y_{n,j}|=
n^{\gamma-\frac{1}{2}}\, \sup_{k\geq n} | {\widetilde Z}_k-{\widetilde Z}_{k+1}|
\leq \sup_{k\geq n}\, k^{\gamma-\frac{1}{2}} r_k \longrightarrow 0.
$$
Hence, if in Theorem \ref{fam_tri_vet_as_inf}, we take $k_n=1$ for
each $n$ and ${\mathcal U}=\bigvee_n{\mathcal F}_n$, then the proof is
concluded.
\end{proof}

\subsection{Proofs of Theorem \ref{thm:Z-no-atomi} and
Theorem \ref{thm:distribution_Z} (Results on the distribution of
$Z_\infty$)}

\indent The proof of Theorem \ref{thm:Z-no-atomi} is a consequence of
the almost sure conditional convergence in Theorem
\ref{thm:asymptotics_Z_tilde}, exactly as shown in
\cite{cri-dai-lou-min}.  \\

To the proof of Theorem \ref{thm:distribution_Z} we premise the
following lemma.

\begin{lem}
If condition \eqref{Z01_Hp} holds, then we have
\begin{equation}\label{x_nmaggiore0}
E\left[\widetilde{Z}_n(1-\widetilde{Z}_n)\right]>0\qquad\forall n\geq0.
\end{equation}
\end{lem}
\begin{proof} For convenience, set
$x_n:=E\left[\widetilde{Z}_n(1-\widetilde{Z}_n)\right]$. We recall that
$\widetilde{Z}_n=N^{-1/2}{\mathbf v}_1^{\top}\mathbf{Z}_n$, where
  $\mathbf{v}_1$ is such that
\begin{equation}\label{cond-v1}
v_{1,k}>0\quad\forall k\quad\mbox{and}\quad N^{-1/2}\sum_{k=1}^N v_{1,k}=1.
\end{equation}
Hence, under assumption \eqref{Z01_Hp}, we immediately get
$x_0>0$. \\
\indent Now, we recall that $(\widetilde{Z}_n)_n$ is a bounded martingale
which satisfies \eqref{eq:dynamic_SA_tilde}, that is
$$
\widetilde{Z}_n=
(1-r_{n-1})\widetilde{Z}_{n-1}+r_{n-1} N^{-1/2}{\mathbf v}_1^{\top}\mathbf{X}_n
$$
with $E[N^{-1/2}{\mathbf v}_1^{\top}\mathbf{X}_n\,|\, {\mathcal F}_{n-1}]=
\widetilde{Z}_{n-1}$.  Therefore we have
$x_n=(E[\widetilde{Z}_0]-E[\widetilde{Z}_n^2])$ for each $n$ and
\begin{equation*}
\begin{split}
\widetilde{Z}_n^2&=
(1-r_{n-1})^2\widetilde{Z}_{n-1}^2
+2(1-r_{n-1})r_{n-1}\widetilde{Z}_{n-1}N^{-1/2}{\mathbf v}_1^{\top}\mathbf{X}_n
+r_{n-1}^2(N^{-1/2}{\mathbf v}_1^{\top}\mathbf{X}_n)^2\\
&\leq
(1-r_{n-1})^2\widetilde{Z}_{n-1}^2
+2(1-r_{n-1})r_{n-1}\widetilde{Z}_{n-1}N^{-1/2}{\mathbf v}_1^{\top}\mathbf{X}_n
+r_{n-1}^2(N^{-1/2}{\mathbf v}_1^{\top}\mathbf{X}_n)
\end{split}
\end{equation*}
Taking the conditional expectation given ${\mathcal F}_{n-1}$, we get
\begin{equation*}
E\left[\widetilde{Z}_n^2\,|\,{\mathcal F}_{n-1}\right]
\leq
(1-r_{n-1}^2)\widetilde{Z}_{n-1}^2+r_{n-1}^2\widetilde{Z}_{n-1},
\end{equation*}
which implies
\begin{equation*}
E[\widetilde{Z}_n^2]\leq
(1-r_{n-1}^2)E[\widetilde{Z}_{n-1}^2]
+r_{n-1}^2E[\widetilde{Z}_{n-1}]
=(1-r_{n-1}^2)E[\widetilde{Z}_{n-1}^2]
+r_{n-1}^2E[\widetilde{Z}_{0}].
\end{equation*}
Therefore, we can conclude by an induction argument on $n$. Indeed, if
$x_{n-1}>0$, i.e. $E[\widetilde{Z}_{n-1}^2]<E[\widetilde{Z}_{0}]$,
then from the above inequality, since $(1-r_{n-1}^2)>0$ by assumption, we
obtain $E[\widetilde{Z}_n^2]<E[\widetilde{Z}_{0}]$, i.e. $x_n>0$.
\end{proof}

We are now ready to prove Theorem  \ref{thm:distribution_Z}.
\\

\noindent {\it Proof of Theorem \ref{thm:distribution_Z}.}  \indent We
recall that $Z_{\infty}$ takes values in $[0,1]$ and
$(\widetilde{Z}_n)_n$ is a bounded martingale which converges a.s. (and
in $L^p$) to $Z_\infty$. Therefore, in particular, setting
$\widetilde{z}_0:=E[\widetilde{Z}_0]$, we have
\begin{equation*}
E[Z_{\infty}]=E[\widetilde{Z}_n]=\widetilde{z}_0\quad
\forall n\qquad\mbox{ and }\qquad
Var[Z_{\infty}]=\lim_{n\rightarrow\infty} Var[\widetilde{Z}_n].
\end{equation*}
Now, as in the proof of the previous Lemma, we set
\begin{equation}\label{def-xn}
x_n:=E\left[\widetilde{Z}_n(1-\widetilde{Z}_n)\right]=
\widetilde{z}_0-\widetilde{z}_0^2-Var[\widetilde{Z}_n]
\end{equation}
and we can state that
$$
P(Z_{\infty} \in\{0, 1\}) = 1 \;\qquad\mbox{if and only if }\qquad\;
E\left[Z_{\infty}(1-Z_{\infty})\right]=
\lim_n x_n=0.
$$
Thus, it is enough to prove that assumption \eqref{Z01_Hp} implies
$\lim_n x_n>0$. To this purpose, we observe that, by
\eqref{eq:dynamic_SA_tilde}, we have
\begin{equation}\label{eq:steps_x_n}
\begin{split}
x_{n+1}&=
\widetilde{z}_0-\widetilde{z}_0^2 - Var[\widetilde{Z}_{n+1}] =
\widetilde{z}_0-\widetilde{z}_0^2 -
E\left[\, Var[\widetilde{Z}_{n+1} \,|\,{\mathcal F}_n]\,\right]
- Var\!\left[\,E[\widetilde{Z}_{n+1} \,|\,{\mathcal F}_n]\,\right]
\\
&=\widetilde{z}_0-\widetilde{z}_0^2 -
\frac{r_n^2}{N}
E\left[\,
E\left[\,
(\mathbf{v}_1^{\top}\Delta\mathbf{M}_{n+1})^2\,|\,{\mathcal F}_n
\,\right]
\,\right]
- Var[\widetilde{Z}_n]
\\
&=x_n - \frac{r_n^2}{N}
E\left[\,
E\left[\,
(\mathbf{v}_1^{\top}\Delta\mathbf{M}_{n+1})^2\,|\,{\mathcal F}_n
\,\right]
\,\right].
\end{split}
\end{equation}
Setting ${\mathbf Y}_n=E[\mathbf{X}_n|\mathcal{F}_{n-1}]=
W^{\top}\mathbf{Z}_n$ (whose components obviously belong to $[0,1]$) and
recalling \eqref{eq:covar-conditional-M} and
\eqref{eq:var-conditional-M}, we obtain
\begin{equation}\label{eq:v1_martingale}
E\left[\,
(\mathbf{v}_1^{\top}\Delta\mathbf{M}_{n+1})^2\,|\,{\mathcal F}_n
\,\right]
\ \!=\!\
\sum_{k=1}^N v_{1,k}^2{Y}_{n,k}(1 - {Y}_{n,k}).
\end{equation}
Now, notice that
\begin{equation*}
N^{-1/2}\mathbf{v}_1^{\top}{\mathbf Y}_n\ =\
N^{-1/2}\mathbf{v}_1^{\top}W^{\top}\mathbf{Z}_n\ =\
N^{-1/2}(W\mathbf{v}_1)^{\top}\mathbf{Z}_n\ =\
N^{-1/2}\mathbf{v}_1^{\top}\mathbf{Z}_n\ =\ \widetilde{Z}_n,
\end{equation*}
and so, for any $k=1,..,N$,
\begin{equation}\label{eq:inequality2_variance_Z_tilde}
N^{-1/2}v_{1,k}{Y}_{n,k}\ =\
\widetilde{Z}_n-N^{-1/2}\sum_{j\neq k}v_{1,j}{Y}_{n,j}\ \leq\
\widetilde{Z}_n.
\end{equation}
Analogously, notice that
\begin{equation*}
N^{-1/2}\mathbf{v}_1^{\top}(\mathbf{1}-{\mathbf Y}_n)\ =\
\left(N^{-1/2}\mathbf{v}_1^{\top}\mathbf{1}\right)\ -\
\left(N^{-1/2}\mathbf{v}_1^{\top}{\mathbf Y}_n\right)\ =
\ 1-\widetilde{Z}_n
\end{equation*}
and so, for any $k=1,..,N$,
\begin{equation}\label{eq:inequality2_variance_1-Z_tilde}
N^{-1/2}v_{1,k}(1-{Y}_{n,k})\ =\
(1-\widetilde{Z}_n)-N^{-1/2}\sum_{j\neq k}v_{1,j}(1-{Y}_{n,j})\ \leq\
1-\widetilde{Z}_n.
\end{equation}
Then, combining~\eqref{eq:inequality2_variance_Z_tilde}
and~\eqref{eq:inequality2_variance_1-Z_tilde},
we get for any $k=1,..,N$
\begin{equation*}
v_{1,k}^2{Y}_{n,k}(1 - {Y}_{n,k})\ \leq\
N \widetilde{Z}_n(1-\widetilde{Z}_n),
\end{equation*}
and hence, recalling~\eqref{def-xn}, \eqref{eq:steps_x_n}
and~\eqref{eq:v1_martingale}, we
obtain
\begin{equation*}
x_{n+1} \geq x_n - Nr_n^2E[\widetilde{Z}_n(1-\widetilde{Z}_n)]
=(1-N r_n^2)x_n.
\end{equation*}
Finally, taking $\bar{n}$ such that
$Nr_n^2 < 1$ for any $n\geq \bar{n}$, we find
\begin{equation*}
x_{n+1}\geq x_{\bar{n}} \prod_{m=\bar{n}}^{n} \left( 1- N r_m^2\right).
\end{equation*}
Hence, since $x_{\bar{n}}>0$ by the previous Lemma and
$\sum_n r_n^2 < +\infty$ for $1/2<\gamma\leq 1$, we
can conclude that $\lim_n x_n>0$.
\qed

\subsection{A CLT for $\widehat{\mathbf{Z}}_{n}$}
\label{subsection_asymptotic_results_Z_hat}

The following result provides a central limit theorem for the
multi-dimensional real stochastic process $(\widehat{\mathbf{Z}}_{n})_n$.

\begin{theo}\label{thm:asymptotics_Z_hat}
We have:
\begin{itemize}
\item[(a)] If $1/2<\gamma<1$, then
\begin{equation}\label{eq:CLT_Z_hat_gamma}
n^{\frac{\gamma}{2}}\,
\widehat{\mathbf{Z}}_{n}\ \longrightarrow
\
\mathcal{N}\left(\ 0\ ,\ Z_{\infty}(1-Z_{\infty})\widehat{\Sigma}_{\gamma}\
\right),
\ \ \ \ stably
\end{equation}
where $\widehat{\Sigma}_{\gamma}$ is defined in~\eqref{def:sigmahat_gamma}.

\item[(b)] If $\gamma=1$ and ${\mathcal R}e(\lambda^{*})<1-(2c)^{-1}$, then
\begin{equation}\label{eq:CLT_Z_hat_less}
\sqrt{n}\,
\widehat{\mathbf{Z}}_{n}\ {\longrightarrow}
\
\mathcal{N}\left(\ 0\ ,\ Z_{\infty}(1-Z_{\infty})\widehat{\Sigma}_1\ \right),
\ \ \ \ stably
\end{equation}
where $\widehat{\Sigma}_1$ is defined in~\eqref{def:sigmahat1}.

\item[(c)] If $\gamma=1$, $\mathcal{R}e(\lambda^{*})=1-(2c)^{-1}$
and~\eqref{eq:condition_r_n_bis_main} holds, then
\begin{equation}\label{eq:CLT_Z_hat_equal}
\sqrt{\frac{n}{\ln(n)}}\,
\widehat{\mathbf{Z}}_{n}\ \longrightarrow
\
\mathcal{N}\left(\ 0\ ,\ Z_{\infty}(1-Z_{\infty})\widehat{\Sigma}_1^*\
\right),
\ \ \ \ stably
\end{equation}
where $\widehat{\Sigma}_1^*$ is defined in~\eqref{def:sigmahat1_star}.
\end{itemize}
\end{theo}

\begin{proof} Set
$\alpha_j=1-\lambda_j= a_j + i\, b_j$ with $\lambda_j\in
  Sp(W)\setminus \{1\}$.  Remember that $a_j>0$ for each $j$ since
  ${\mathcal Re}(\lambda_j)<1$ for each $j$. Moreover recall the
  definition of the matrices $U,\,V$ and $D$ given in Section
  \ref{section_model} and in Subsection
  \ref{subsection_preliminaries}.\\

\noindent From dynamics \eqref{eq:dynamic_SA_hat}, we get
$$
\widehat{\mathbf{Z}}_{n+1}=
\left[I-r_nU(I-D)V^{\top}\right]\widehat{\mathbf{Z}}_{n}+
r_n\, UV^{\top}\,\Delta\mathbf{M}_{n+1}
=
U\left[I-r_n(I-D)\right]V^{\top}\widehat{\mathbf{Z}}_{n}+
r_n\, UV^{\top}\,\Delta\mathbf{M}_{n+1},
$$
where the identity matrices adopted above have different dimensions.
\noindent where the last equality holds
because relations \eqref{eq:relazioni-2}
imply $UV^{\top}\,\widehat{\mathbf{Z}}_{n}=UV^{\top}\,\mathbf{Z}_{n}=
\widehat{\mathbf{Z}}_{n}$.  Therefore, if we take $m_0$ large enough such
that $a_j r_n<1$ for $n\geq m_0$ and all $j$, we can write
\begin{equation}\label{eq-z-hat-1}
\widehat{\mathbf{Z}}_{n+1}=
C_{m_0,n}\widehat{\mathbf{Z}}_{m_0}
+
\sum_{k=m_0}^n C_{k+1,n}\,r_k UV^{\top}\,\Delta{\mathbf M}_{k+1},
\end{equation}
with
$$C_{k+1,n}=\prod_{m=k+1}^n\{U [I-r_m(I-D)] V^{\top}\}.$$ For the sequel, it
is important to note that $C_{k+1,n}$ is a real matrix since, by
\eqref{eq:relazioni-2} and \eqref{eq:decomposition-matrix} it is
equivalent to a product of real matrices, i.e. $(UV^{\top})-r_m(UV^{\top}+{\mathbf
  u}_1{\mathbf v}_1^{\top}- W^{\top})=(UV^{\top})-r_m(I-W^{\top})$.  Moreover, using
relations~\eqref{eq:relazioni-2} again, we get
\begin{equation}\label{matrice-C}
C_{k+1,n}=U A_{k+1,n} V^{\top},
\end{equation}
where $A_{k+1,n}$ is the diagonal matrix given by
\begin{equation*}
[A_{k+1,n}]_{j,j}=
\begin{cases}
\prod_{m=k+1}^n \left(1-\alpha_j r_m\right)\quad
&\mbox{for } m_0-1\leq k\leq n-1\\
1\quad&\mbox{for } k=n.
\end{cases}
\end{equation*}
Observe that we have
$$
[A_{k+1,n}]_{j,j}=\frac{p_{n,j}}{p_{k,j}}=
\frac{\ell_{k,j}}{\ell_{n,j}}\quad\mbox{for } m_0-1\leq k\leq n,
$$
with
$$
p_{m_0-1,n}=\ell_{m_0-1,n}=1, \qquad
p_{k,j}=\prod_{m=m_0}^k \left(1-\alpha_j r_m\right),\qquad
\ell_{k,j}=p_{k,j}^{-1}
\quad\mbox{for } m_0\leq k\leq n.
$$
Finally, notice that, since $C_{k+1,n}UV^{\top}=C_{k+1,n}$ by relations
\eqref{eq:relazioni-2} and \eqref{matrice-C}, we can rewrite
\eqref{eq-z-hat-1} as
$$
\widehat{\mathbf{Z}}_{n+1}=
C_{m_0,n}\widehat{\mathbf{Z}}_{m_0}
+
\sum_{k=m_0}^n {\mathbf T}_{n,k},\qquad\mbox{where}
\qquad {\mathbf T}_{n,k}=r_k C_{k+1,n} \Delta{\mathbf M}_{k+1}.
$$
We will establish the asymptotic behavior of
$\widehat{\mathbf{Z}}_{n}$ by studying separately the terms
$C_{m_0,n}\widehat{\mathbf{Z}}_{m_0}$ and $\sum_{k=m_0}^n {\mathbf
  T}_{n,k}$. \\

Concerning the first term, note that
by \eqref{affermazione1} in Lemma \ref{lemma-tecnico_1},
we have that, for any $\epsilon\in (0,1)$,
\begin{equation}\label{affermazione1_thm}
|C_{m_0,n}\widehat{\mathbf{Z}}_{m_0}|=
O\left(|p_{n}^*|\right)=
\begin{cases}
O\left(
\exp \left[ -(1-\epsilon)\frac{ca^* }{1-\gamma}n^{1-\gamma} \right]
\right) & \mbox{if } 1/2<\gamma<1
\\
O\left(\ n^{- (1-\epsilon)ca^*}\ \right)&
\mbox{if } \gamma=1,
\end{cases}
\end{equation}
where the symbol $^*$ refers to quantities $a_j$ and $p_{n,j}$
corresponding to $\lambda_j=\lambda^*\in \lambda_{\max}(D)$.
Therefore, for the case (a) (i.e. $1/2<\gamma<1$) and (b)
(i.e. $\gamma=1$ and $\mathcal{R}e(\lambda^{*})<1-(2c)^{-1}$), we
have
\begin{equation}
|C_{m_0,n}\widehat{\mathbf{Z}}_{m_0}|=o(n^{-\gamma/2}).
\end{equation}
Indeed, this fact follows immediately for $1/2<\gamma<1$ and, for $\gamma
=1$ one has to note that, since we assume
$\mathcal{R}e(\lambda^{*})<1-(2c)^{-1}$, that is $ca^*>1/2$, we can
choose $\epsilon$ small enough so that $(1-\epsilon)ca^*> 1/2$.
Moreover, for the case (c) (i.e. $\gamma=1$ and
$\mathcal{R}e(\lambda^{*})=1-(2c)^{-1}$, i.e. $ca^*=1/2$), since we
assume condition~\eqref{eq:condition_r_n_bis_main}, by
\eqref{affermazione1_bis} in Lemma \ref{lemma-tecnico_1}, we have
\begin{equation}\label{affermazione1_thm_bis}
|C_{m_0,n}\widehat{\mathbf{Z}}_{m_0}|\ =\
O\left(|p^*_{n}|\right)\ =\ O\left(n^{-c a^*}\right)\ =\
O\left(n^{-\frac{1}{2}}\right).
\end{equation}
Therefore, if we set
\begin{equation}
t_n=\begin{cases}
n^{\frac{\gamma}{2}}\quad &\mbox{for case (a)}
\\[3pt]
n^{\frac{1}{2}}\quad &\mbox{for case (b)}
\\[3pt]
\left(n/\ln(n)\right)^{\frac{1}{2}}\quad &\mbox{for case (c)},
\end{cases}
\end{equation}
then we obtain $t_n |C_{m_0,n}\widehat{\mathbf{Z}}_{m_0}|\to 0$ almost
surely.\\

We now focus on the asymptotic behavior of the second term.
Specifically, we aim at pro\-ving that $t_n\sum_{k=m_0}^n{\mathbf
  T}_{n,k}$ converges stably to a suitable Gaussian kernel.  For this
purpose, we set ${\mathcal G}_{n,k}={\mathcal F}_{k+1}$, and consider
Theorem~\ref{thm:triangular} (recall that ${\mathbf T}_{n,k}$ are real
random vectors). Given the fact that condition $(c1)$ of
Theorem~\ref{thm:triangular} is obviously satisfied, we will check
conditions $(c2)$ and $(c3)$.\\

Regarding condition $(c2)$, we observe that
\begin{equation*}
\begin{split}
\sum_{k=1}^n (t_n\mathbf{T}_{n,k})(t_n\mathbf{T}_{n,k})^{\top}
&= t_n^2\sum_{k=1}^n r_k^2
C_{k+1,n}(\Delta{\mathbf M}_{k+1})(\Delta{\mathbf M}_{k+1})^{\top}C_{k+1,n}^{\top}
\\
&= U \left(t_n^2\sum_{k=1}^n r_k^2
A_{k+1,n}\,V^{\top}\,(\Delta{\mathbf M}_{k+1})(\Delta{\mathbf M}_{k+1})^{\top}\,
V\,A_{k+1,n}
\right) U^{\top}.
\end{split}
\end{equation*}
Therefore, it is enough to study the convergence of
\begin{equation*}
t_n^2\sum_{k=1}^n r_k^2
A_{k+1,n}\,V^{\top}\,(\Delta{\mathbf M}_{k+1})(\Delta{\mathbf M}_{k+1})^{\top}\,
V\,A_{k+1,n}.
\end{equation*}
To this purpose, we set $B_{k+1,h,j}=[V^{\top}\,(\Delta{\mathbf
  M}_{k+1})(\Delta{\mathbf M}_{k+1})^{\top}\,V]_{h,j}$ and observe that an element
of the above matrix is of the form
$$
t_n^2\sum_{k=1}^{n} r_k^2 [A_{k+1,n}]_{h,h}B_{k+1,h,j}[A_{k+1,n}]_{j,j}\ =\
t_n^2p_{n,h}p_{n,j}
\sum_{k=1}^{n-1} r_k^2\ell_{k,h}\ell_{k,j}
B_{k+1,h,j}
+ t_n^2r_n^2B_{n+1,h,j},
$$ where $t_n^2r_n^2B_{n+1,h,j}=O(t_n^2r_n^2)\to 0$. We now fix $j$
and $h$ and apply Lemma \ref{lemma-serie-rv} to the first addend in the
above equality. Indeed, this quantity can be written as
$v_n\sum_{k=m_0}^{n-1} Y_k/c_kv_k$, where
$$
Y_n=B_{n+1,h,j},\quad
c_n=\frac{1}{t_n^2r_n^2} > 0
\quad\mbox{and }\quad
v_n=t_n^2p_{n,h}p_{n,j}\in{\mathbb C}\setminus\{0\}
$$
satisfy the assumptions of Lemma \ref{lemma-serie-rv}. More precisely,
setting ${\mathcal H}_n={\mathcal F}_{n+1}$, by
\eqref{eq:limite-conditional-M}, we have
\begin{equation*}
\begin{split}
E[Y_n\,|\,{\mathcal H }_{n-1}] =
E[B_{n+1,h,j}\, |\, {\mathcal F}_n] =
\left[V^{\top}
E[(\Delta{\mathbf M}_{n+1})(\Delta{\mathbf M}_{n+1})^{\top}\,|\,{\mathcal F}_n]
V\right]_{h,j}
\ \stackrel{a.s}\longrightarrow\
(\mathbf{v}_{h}^{\top}\mathbf{v}_{j})Z_{\infty}(1-Z_{\infty})
\end{split}
\end{equation*}
and, moreover, we have
$$
\sum_{n}\frac{ E[\,|Y_n|^2] }{ c_n^2 }=\sum_n E[\,|Y_n|^2]r_n^4 t_n^4
=\sum_n r_n^4 O(n^{2\gamma})=\sum_n O(1/n^{2\gamma})<+\infty.
$$
In addition, as we have observed above, $|v_n|=t_n^2|p_{n,h}
p_{n,j}|=t_n^2 O(|p_{n}^*|^2)\to 0$ and, by
\eqref{affermazione2} in Lemma \ref{lemma-tecnico_2} and
\eqref{affermazione2-log} in Lemma \ref{lemma-tecnico_3}, we have
\begin{equation*}
\lim_n v_n\sum_{k=m_0}^n \frac{1}{c_kv_k}=
\begin{cases}
\frac{c}{\alpha_h+\alpha_j }\;&\mbox{if } 1/2<\gamma<1
\\[3pt]
\frac{c^2}{c(\alpha_h+\alpha_j)-1}\;&\mbox{if } \gamma=1,\ c(a_h+a_j)>1
\\[3pt]
0\;&\mbox{if } \gamma=1,\; c(a_h+a_j)=1,\; c(\alpha_h+\alpha_j)\neq 1
\;\mbox{and \eqref{eq:condition_r_n_bis_main} holds}
\\[3pt]
c^2\;&\mbox{if } \gamma=1,\; c(\alpha_h+\alpha_j)=1\;
\mbox{and \eqref{eq:condition_r_n_bis_main} holds}.
\end{cases}
\end{equation*}

Finally, we have $|v_n|\sum_{k=m_0}^n\frac{1}{c_k |v_k|}=O(1)$ by
\eqref{affermazione3} in Lemma \ref{lemma-tecnico_2} and
\eqref{affermazione3-log} in Lemma \ref{lemma-tecnico_3} (with $u=1$),
and, using \eqref{conti} and \eqref{conti_log} in Appendix, we get
\begin{equation*}
c_n |v_n| \left|\frac{1}{v_n}-\frac{1}{v_{n-1}}\right|
=
\frac{1}{r_n^2\, |\ell_{n,h}\ell_{n,j}|}
\left|
\frac{\ell_{n,h}\ell_{n,j}}{t_n^2}-\frac{\ell_{n-1,h}\ell_{n-1,j}}{t_{n-1}^2}
\right|
=O(1).
\end{equation*}
Hence, recalling Remark \ref{rem-serie-rv}, also the last condition
required in Lemma \ref{lemma-serie-rv} is verified.
\\

Regarding condition $(c_3)$, we observe that, using the inequalities
$$
|{\mathbf T}_{n,k}|=r_k |C_{k+1,n}\Delta{\mathbf M}_{k+1}|
\leq
r_k |U|\,|A_{k+1,n}|\,|V^{\top}|\,|\Delta{\mathbf M}_{k+1}|
\leq
K r_k |A_{k+1,n}|,
$$
with a suitable constant $K$, we find for any $u>1$
\begin{equation*}
\begin{split}
\left(\sup_{m_0\leq k\leq n}|t_n\mathbf{T}_{n,k}|\right)^{2u}
&\leq
t_n^{2u} \sum_{k=m_0}^{n-1} |\mathbf{T}_{n,k}|^{2u}
+ t_n^{2u} |\mathbf{T}_{n,n}|^{2u}
=
t_n^{2u} O\left(
|p_{n}^*|^{2u}\sum_{k=m_0}^{n-1} r_k^{2u}\, |\ell_{k}^*|^{2u}
\right)
\!+t_n^{2u} O(r_n^{2u}),
\end{split}
\end{equation*}
where, for the last equality, we have used \eqref{affermazione1_bis}
and \eqref{affermazione1_bis-ell} in Lemma \ref{lemma-tecnico_1}. Now,
by \eqref{affermazione3} in Lemma \ref{lemma-tecnico_2} and
\eqref{affermazione3-log} in Lemma \ref{lemma-tecnico_3} (with
$\alpha_1=\alpha_2=\alpha^*=1-\lambda^*$ and $u>1$), we have
\begin{equation*}
\begin{split}
O\left(
|p_{n}^*|^{2u}\sum_{k=m_0}^{n-1} r_k^{2u}\, |\ell_{k}^*|^{2u}
\right)
=
\begin{cases}
O(\,n^{-\gamma(2u-1)}\,)\quad &\mbox{if } 1/2<\gamma<1\\
O(\,n^{-(2u-1)}\,)\quad &\mbox{if }\gamma=1,\ 2uca^*>2u-1\\
O(\,n^{-u}\,)\quad &\mbox{if } \gamma=1, \ 2ca^*=1\
\mbox{and \eqref{eq:condition_r_n_bis_main} holds}.
\end{cases}
\end{split}
\end{equation*}
Therefore, for cases (a) and (c), it is immediate to obtain
$$
t_n^{2u} O\left(
|p_{n}^*|^{2u}\sum_{k=m_0}^{n-1} r_k^{2u}\, |\ell_{k}^*|^{2u}
\right)
+t_n^{2u} O(r_n^{2u})\longrightarrow 0
$$ for any $u>1$; while in case (b) in order to have the above
convergence to zero, we have to choose $u>1$ such that $2uca^*>2u-1$,
i.e. $2u(ca^*-1)+1>0$. This choice is always possible: indeed, or
$ca^*-1\geq 0$ and so we can take any $u>1$, or $ca^*-1<0$ and we have
to take $u\in (1, (2-2ca^*)^{-1})$ (note that $(2-2ca^*)^{-1}>1$ since
$2ca^*>1$ by assumption). As a consequence of the above convergence to
zero, we obtain condition (c3) of Theorem~\ref{thm:triangular}.  \\

Summing up, all the conditions required by
Theorem~\ref{thm:triangular} are satisfied and so we can apply this
theorem and obtain the stable convergence of $t_n\sum_{k=m_0}^n
\mathbf{T}_{n,k}$ to the Gaussian kernel with random variance defined
in Theorem~\ref{thm:asymptotics_Z_hat} for each of the three cases.
\end{proof}

\noindent {\it Proof of Theorem \ref{thm:real_positive_variances}.}
\indent First, consider case (a), i.e. $1/2<\gamma<1$, and recall the
definition of $\widehat{S}_{\gamma}$ in~\eqref{def:sigmahat_gamma}.
Then, since $l_j:=(1-\lambda_j)$, for $\lambda_j\in Sp(W)\setminus
\{1\}$, have positive real parts by the assumptions on~$W$, we have
\begin{equation*}
[\widehat{S}_{\gamma}]_{h,j}\ =\ (\mathbf{v}_{h}^{\top}\mathbf{v}_{j})
\frac{c}{l_h+l_j}\
=\ (\mathbf{v}_{h}^{\top}\mathbf{v}_{j})c\int_{0}^{\infty}\exp[-u(l_h+l_j)]du.
\end{equation*}
Then, setting $L:=(I-D)$ and $M(u):=U\exp(-uL)V^{\top}$ for
$u\in(0,+\infty)$, we can write
$$\widehat{\Sigma}_{\gamma}\ =\ U\widehat{S}_{\gamma}U^{\top}
\ =\ c\int_{0}^{\infty}M(u)M^{\top}(u)du.$$
Notice that, for any $u\in(0,+\infty)$, the matrix $M(u)$ has real
entries since, by \eqref{eq:relazioni-2} and \eqref{eq:decomposition-matrix},
we have
$$
M(u)\ =\ U\sum_{k=0}^{+\infty}\frac{(-uL)^k}{k!}V^{\top}\ =\
UV^{\top}+\sum_{k=1}^{+\infty}\frac{(-uULV^{\top})^k}{k!}\ =\
\exp[-u(I-W^{\top})]-\mathbf{u}_1\mathbf{v}_1^{\top}.
$$ Moreover, for any $u\in(0,+\infty)$, the matrix $M(u)$ has rank
$(N-1)$ and $M^{\top}(u)\mathbf{v}_1=\mathbf{0}$ by
\eqref{eq:relazioni-2}. Therefore, for any $u\in(0,+\infty)$, the
matrix~$M(u)M^{\top}(u)$ is a positive semi-definite real matrix with
rank~$(N-1)$ and $M(u)M^{\top}(u)\mathbf{v}_1=\mathbf{0}$ (see
\cite[Observation 7.1.8]{horn_johnson}).  These facts imply the first
part of statement~(a).  For the second part, denoting by
$\mathbf{e}_j$ the vector such that $e_{j,j}=1$ and $e_{j,h}=0$ for
all $h\neq j$, we have for $j\neq k$
$$\Sigma_{\gamma,j,k}=
(\mathbf{e}_j-\mathbf{e}_k)^{\top}\widehat{\Sigma}_{\gamma}
(\mathbf{e}_j-\mathbf{e}_k)
$$
and so $\Sigma_{\gamma,j,k}=0$ if and only if
$M(u)^{\top}(\mathbf{e}_j-\mathbf{e}_k)=0$ for almost every $u\in
(0,+\infty)$. But this is not possible since $Ker(M(u)^{\top})$ is generated
by $\mathbf{v}_1$, which has all the entries strictly greater than
zero as pointed out in Section \ref{section_model}. This concludes the
proof of case (a).
\\

\indent The proof of case (b) is analogous, by setting
$l_j:=(1-\lambda_j-(2c)^{-1})$ for $\lambda_j\in Sp(W)\setminus
\{1\}$, which have positive real parts by condition ${\mathcal
  Re}(\lambda^{*})<1-(2c)^{-1}$, and $L:=(I-D-I(2c)^{-1})$.
\\

\indent For the proof of case (c), i.e. $\gamma=1$ and
$\mathcal{R}e(\lambda^*)=1-(2c)^{-1}$, let $1\leq q\leq (N-1)$ be the
number of distinct eigenvalues $\lambda_j=a_j+ib_j\in
Sp(W)\setminus\{1\}$ and, for any $1\leq h\leq q$, let
$U_h$ and $V_h$ be the sub-matrices of $U$ and
$V$ whose columns are, respectively, the left and the right eigenvectors
associated to $\lambda_h$. Then, by the
definition of $\widehat{S}_{1}^*$ in~\eqref{def:sigmahat1_star}, we
have
$$
\widehat{\Sigma}_1^*\ =\
U\widehat{S}_1^*U^{\top}\ =\
\sum_{1\leq h,j\leq q}
U_hV_h^{\top}V_jU_j^{\top}
\mathbbm{1}_{\{\lambda_h+\lambda_j=2-c^{-1}\}}.
$$
Then, since
$$\{\lambda_h+\lambda_j=2-c^{-1}\}\ =\
\left(\{a_h=1-(2c)^{-1}\}\cap\{a_j=1-(2c)^{-1}\}\cap\{b_h=-b_j\}\right),
$$ setting the $N\times N$-matrix $M_h:=U_hV_h^{\top}$ for any $1\leq
h\leq q$ and denoting by $1\leq p\leq q$ the number of distinct
eigenvalues $\lambda_j\in A^*$, we can write
$$ \sum_{1\leq h,j\leq
  q}M_hM_j^{\top}\mathbbm{1}_{\{\lambda_h+\lambda_j=2-c^{-1}\}}\ =\ \sum_{1\leq
  h,j\leq p}M_hM_j^{\top}\mathbbm{1}_{\{b_h=-b_j\}}\ =\ \sum_{1\leq
  h\leq p}M_hM_{j(h)}^{\top}\ ,$$ where $j(h)$ indicates the index
$1\leq j\leq p$ such that $b_j=-b_h$.  Notice that, since $W$ has real
entries, for any non-real $\lambda_h\in Sp(W)$, there exists
$\lambda_j\in Sp(W)$ such that $\lambda_j=\overline{\lambda}_h$;
moreover, $\overline{\mathbf{u}}_h$ and $\overline{\mathbf{v}}_h$ are
respectively left and right eigenvectors associated to $\lambda_j$.
Hence, denoting by $T$ the non-singular matrix such that
$U_j=\overline{U}_hT$ and $V_j=\overline{V}_h(T^{\top})^{-1}$, we have
that
$$M_j\ =\ U_{j}V_{j}^{\top}\ =\ \overline{U}_hTT^{-1}\overline{V}_h^{\top}\ =\
\overline{U}_h\overline{V}_h^{\top}\ =\ \overline{M}_h.$$
Thus, we have
$$\sum_{1\leq h\leq p}M_hM_{j(h)}^{\top}\ =\ \sum_{1\leq h\leq
  p}M_h\overline{M}_h^{\top},$$ which is a positive semi-definite
matrix of rank $m^*$ (see \cite[Observation 7.1.8]{horn_johnson}).

Concerning the second part of case (c), since
$$\Sigma_{1,j,k}^*\ = \ (\mathbf{e}_j-\mathbf{e}_k)^{\top}
\widehat{\Sigma}_{1}^*(\mathbf{e}_j-\mathbf{e}_k),$$ we have that
$\Sigma_{1,j,k}^*=0$ if and only if $(\mathbf{e}_j-\mathbf{e}_k)\in
Ker(\widehat{\Sigma}_1^*)$.  Now, notice that
$Ker(\widehat{\Sigma}_1^*)=\bigcap_{1\leq h\leq p} Ker(U_h^{\top})$, and
hence $Ker(\widehat{\Sigma}_1^*)$ is generated by
$\{\mathbf{v}_j,j:\lambda_j\notin A^*\}$.  Finally, since the
following decomposition holds
$$(\mathbf{e}_j-\mathbf{e}_k)\ =\ \sum_{h=1}^N(\mathbf{u}_h^{\top}
(\mathbf{e}_j-\mathbf{e}_k))\mathbf{v}_h\ =\
\sum_{h=1}^N(u_{h,j}-u_{h,k})\mathbf{v}_h,$$
it is enough to have $u_{h,j}\neq u_{h,k}$ for some $h$ such that
$\lambda_h\in A^*$ to prove that $\Sigma_{1,j,k}^*>0$.  Analogously,
we can prove that $[\widehat{\Sigma}_1^*]_{jj}>0$ when $u_{h,j}\neq 0$
for some $h$ such that $\lambda_h\in A^{*}$.  This concludes the proof
of case (c).
\qed

\subsection{Proofs of Theorem \ref{thm:asymptotics_Z} and Theorem
\ref{thm:asymptotics_Z_j_Z_k} (CLTs for ${\mathbf Z}_n$)}
\label{subsection_proof_Z_n}

Let us remind the decomposition~\eqref{eq:decomposition_Z}, i.e.
$$
\mathbf{Z}_{n}=
\widetilde{Z}_n\mathbf{1}+\widehat{\mathbf{Z}}_{n}.
$$
Hence, the asymptotic behavior of the process $(\mathbf{Z}_{n})_n$ can
be obtained by combining the asymptotic results concerning
$(\widetilde{Z}_{n})_n$ and $(\widehat{\mathbf{Z}}_{n})_n$ established
in the previous subsections. As we have already seen, we have the almost sure
synchronization, i.e.
$$
{\mathbf Z}_n\stackrel{a.s.}\longrightarrow Z_\infty {\mathbf 1}.
$$
Moreover, from Theorem \ref{thm:asymptotics_Z_tilde}, we easily obtain
for $1/2<\gamma\leq 1$
\begin{equation*}
n^{\gamma-\frac{1}{2}}
\left(\widetilde{Z}_{n}-Z_{\infty}\right)\,{\mathbf 1}\
{\longrightarrow}\
\mathcal{N}
\left(
\ 0\ ,\
\widetilde{\sigma}_{\gamma}^2\,Z_{\infty}(1-Z_{\infty}){\mathbf 1}{\mathbf 1}^{\top}\
\right),\
\ \ \ \mbox{stably in the strong sense},
\end{equation*}
and we recall the central limit theorem for the multi-dimensional
process $\widehat{\mathbf{Z}}_{n}$ presented in
Theorem~\ref{thm:asymptotics_Z_hat}.\\

Hence, for Theorem \ref{thm:asymptotics_Z}(a), we observe that
$$
n^{\gamma - 1/2}({\mathbf Z}_n-Z_{\infty}{\mathbf 1})\ =\
n^{\gamma - 1/2}(\widetilde{Z}_n-Z_{\infty}){\mathbf 1}\ +\
\frac{1}{n^{(1-\gamma)/2}}(n^{\gamma/2}\,\widehat{\mathbf Z}_n\,),
$$
where the first term converges stably to a Gaussian kernel and the
second one converges in probability to zero.
\\

For Theorem \ref{thm:asymptotics_Z}(b), we observe that
$$
\sqrt{n}({\mathbf Z}_n-Z_{\infty}{\mathbf 1})\ =\
\sqrt{n}(\widetilde{Z}_n-Z_{\infty}){\mathbf 1}\ +\
\sqrt{n}\,\widehat{\mathbf Z}_n\,,
$$ where the first term converges to a Gaussian kernel stable in the
strong sense and the second one converges stably to a Gaussian
kernel. Since $\widehat{\mathbf{Z}}_{n}$ is
$\mathcal{F}_n$-measurable, by applying Theorem \ref{blocco}, we
can conclude.  \\

For Theorem \ref{thm:asymptotics_Z}(c), we observe that
$$
\frac{\sqrt{n}}{\sqrt{\ln(n)}}({\mathbf Z}_n-Z_{\infty}{\mathbf 1})\ =\
\left(\frac{1}{\sqrt{\ln(n)}}\right)\sqrt{n}(\widetilde{Z}_n-Z_{\infty})
{\mathbf 1}\ +\
\frac{\sqrt{n}}{\sqrt{\ln(n)}}\,\widehat{\mathbf Z}_n\,,
$$ where the first term converges in probability to zero and the
second one converges stably to a Gaussian kernel. Thus Theorem
\ref{thm:asymptotics_Z} is proven.  \\

Finally, we observe that
$$
Z_{n,j}-Z_{n,k}=\widehat{Z}_{n,j}-\widehat{Z}_{n,k}.
$$
Therefore, Theorem \ref{thm:asymptotics_Z_j_Z_k} immediately
follows from the central limit theorem for the $N$-dimensional process
$(\widehat{\mathbf Z}_n)_n$.

\section{Examples of Weighted Adjacency Matrices}\label{section_examples}

In this section, we analyze in detail the results presented in
Section~\ref{section_asymptotic_results} for some interesting examples
of weighted adjacency matrices.

\subsection{``Mean-field'' interaction}\label{subsection_medione}

This kind of interaction can be expressed in terms of a particular
weighted adjacency matrix $W$ as follows: for any $1\leq j, k\leq N$
\begin{equation}\label{def:W_crimaldi}
w_{j,k}\ =\ \frac{\alpha}{N}\ +\ \delta_{j,k}(1-\alpha)
\qquad\mbox{with } \alpha\in [0,1],
\end{equation}
where $\delta_{j,k}$ is equal to $1$ when $j=k$ and to $0$
otherwise. Note that $W$ in~\eqref{def:W_crimaldi} is irreducible for
$\alpha>0$. Since $W$ is doubly stochastic, we have (see
Remark~\ref{rem:Z_tilde_doubly_stoch}) ${\mathbf v}_1={\mathbf
  u}_1=N^{-1/2}\mathbf{1}$ and so (i) the random variable
$\widetilde{Z}_{n}$ coincides with the average of the processes
$Z_{n,k}$, i.e. $N^{-1}\mathbf{1}^{\top}\mathbf{Z}_{n}$, (ii)
$\widehat{\mathbf{Z}}_{n}=
\left(I-N^{-1}\mathbf{1}\mathbf{1}^{\top}\right)\mathbf{Z}_n$ and
(iii) $\widetilde{\sigma}_{\gamma}^2=\frac{c^2}{N(2\gamma-1)}$ for
$1/2<\gamma\leq 1$.  Furthermore, we have $\lambda_j=1-\alpha$ for all
$\lambda_j\in Sp(W)\setminus\{1\}$ and, consequently, the conditions
${\mathcal R}e(\lambda^*)<1-(2c)^{-1}$ or ${\mathcal
  R}e(\lambda^*)=1-(2c)^{-1}$ required in the previous results when
$\gamma=1$ correspond to the conditions $2c\alpha>1$ or $2c\alpha=1$.
Finally, since $W$ is also symmetric, we have $U=V$ and so
$U^{\top}U=V^{\top}V=I$ and
$UU^{\top}=VV^{\top}=I-N^{-1}\mathbf{1}\mathbf{1}^{\top}$. We thus
obtain:
\begin{itemize}
\item[(a)] for $1/2<\gamma<1$, $\widehat{S}_\gamma=\frac{c}{2\alpha}I$
  and
  $\widehat{\Sigma}_{\gamma}=
\frac{c}{2\alpha}(I-N^{-1}\mathbf{1}\mathbf{1}^{\top})$;

\item[(b)] for $\gamma=1$ and $2c\alpha>1$,
  $\widehat{S}_1=\frac{c^2}{2c\alpha-1}I$ and
  $\widehat{\Sigma}_{1}=\frac{c^2}{2c\alpha-1}
(I-N^{-1}\mathbf{1}\mathbf{1}^{\top})$;

\item[(c)] for $\gamma=1$ and $2c\alpha=1$, $\widehat{S}_1^*=c^2 I$ and
$\widehat{\Sigma}^*_{1}=c^2(I-N^{-1}\mathbf{1}\mathbf{1}^{\top})$.
\end{itemize}
Therefore, our theorems contain as particular cases part of the
results proven in \cite{cri-dai-lou-min, cri-dai-min, dai-lou-min}.
However, differently from these papers, we do not deal with the cases
$2c\alpha < 1$ or $\gamma\leq 1/2$, which are still interesting for
synchronization phenomena but lead to quite different asymptotic
results. We have already discussed the case $\gamma\leq 1/2$ in
Section \ref{section_model} and, regarding the case $\gamma=1$ and $0<
2c\alpha < 1$, we recall that in \cite{cri-dai-min} it has been
determined the rate of synchronization, but not the asymptotic
distribution.

\subsection{``Cycle'' interaction}\label{subsection_ciclo}

Another possible scenario consists in a graph in which the vertices
forms a circle and each one influences only the vertex at his right
side.  This interaction can be modeled by using the adjacency matrix
$W$ defined as follows: for any $1\leq j\leq (N-1)$ and $1\leq k\leq
N$ we have
\begin{equation}\label{def:W_ciclo}
w_{j,k}\ =\ \left\{
\begin{aligned}
&1 \qquad\qquad\mbox{ if } k=j+1, \\
&0 \qquad\qquad\mbox{ otherwise},
\end{aligned}
\right.
\end{equation}
while, for $j=N$, we have $w_{N,1}=1$ and $w_{N,k}=0$ for any $2\leq k
\leq N$. Since $W$ is again doubly stochastic, we have
$\mathbf{u}_1=\mathbf{v}_1=N^{-1/2}\mathbf{1}$, which implies (i)
$\widetilde{Z}_{n}=N^{-1}\mathbf{1}^{\top}\mathbf{Z}_{n}$, (ii)
$\widehat{\mathbf{Z}}_{n}=\left(I-N^{-1}\mathbf{1}\mathbf{1}^{\top}\right)
\mathbf{Z}_n$ and (iii)
$\widetilde{\sigma}_{\gamma}^2=\frac{c^2}{N(2\gamma-1)}$ for
$1/2<\gamma\leq 1$ as in
Subsection~\ref{subsection_medione}. Moreover, it is easy to verify
that in this case the eigenvalues of $W$ are $\lambda_1=1$ and
$\lambda_j=\exp[i(j-1)2\pi/N]$, for $j=2,..,N$.  Hence, since in this
case $\mathcal{R}e(\lambda^*)=\cos(2\pi/N)$, conditions ${\mathcal
  R}e(\lambda^*)<1-(2c)^{-1}$ or ${\mathcal
  R}e(\lambda^*)=1-(2c)^{-1}$ required in the previous results when
$\gamma=1$ correspond to the conditions $2c(1-\cos(2\pi/N))>1$ or
$2c(1-\cos(2\pi/N))=1$.  Moreover, for each $\lambda_j\in
Sp(W)\setminus\{1\}$, the $k^{th}$ element of the corresponding left
and right eigenvectors are, respectively,
$u_{j,k}=N^{-1/2}\exp[-i(j-1)k2\pi/N]$ and
$v_{j,k}=N^{-1/2}\exp[i(j-1)k2\pi/N]$. Therefore, since we have the
analytic expressions of $U$ and $V$, it is possible to compute the
asymptotic variance-covariance matrices according to the size $N$ of
the network and their eigenvalues and eigenvectors.  For instance, for
$N=4$ we have:

\begin{itemize}
\item[(a)] for $1/2<\gamma<1$,
the non-zero eigenvalues of
$\widehat{\Sigma}_{\gamma}$ are $c/2$, $c/2$, $c/4$,
with the corresponding eigenvectors $(-1,0,1,0)$, $(0,-1,0,1)$, $(-1,1,-1,1)$;

\item[(b)] for $\gamma=1$ and $c>1/2$,
the non-zero eigenvalues of
$\widehat{\Sigma}_1$ are $c^2(2c-1)^{-1}$, $c^2(2c-1)^{-1}$, $c^2(4c-1)^{-1}$,
with the corresponding eigenvectors $(-1,0,1,0)$, $(0,-1,0,1)$, $(-1,1,-1,1)$;

\item[(c)] for $\gamma=1$ and $c=1/2$,
the non-zero eigenvalue of
$\widehat{\Sigma}_1^*$ is $1/4$ with multiplicity two and
the corresponding eigenvectors are $(-1,0,1,0)$, $(0,-1,0,1)$.
\end{itemize}

\subsection{``Special vertex'' case}\label{subsection_caso_non_bistocastico}

In the previous two examples, the matrix $W$ is doubly stochastic
(also symmetric in the first example).  As a different situation, we
may consider the case in which there exists a ``special vertex'' whose
influence on the graph is different with respect to the one of all the
other elements in the system.  This interactive structure can be
expressed in terms of a particular adjacency matrix defined as follows:
\begin{equation}\label{def:W_caso_non_bistocastico}
W=\mathbf{a}_p\mathbf{1}^{\top},\qquad\mbox{ with }\qquad
\mathbf{a}_p\ :=\ \left(\ p\ ,\ \frac{1-p}{N-1}\ ,\ ...\ ,\ \frac{1-p}{N-1}\
\right)^{\top},
\end{equation}
where $0<p<1$ is a weight that represents how much any vertex of
the system is influenced by the ``special vertex''.  Notice that
$\sum_{i=1}^Na_{p,i}=1$ for any $0<p<1$. Moreover, we have
$\mathbf{v}_1=\mathbf{a}_pN^{1/2}$ and hence
$UV^{\top}=I-\mathbf{u}_1\mathbf{v}_1^{\top}=I-\mathbf{1}\mathbf{a}_p^{\top}$ and
\begin{equation*}
\widetilde{\sigma}_{\gamma}^2=
\frac{c^2}{N}\, \frac{\|\mathbf{v}_1\|^2}{(2\gamma-1)}\ =\
\frac{c^2}{(2\gamma-1)}\|\mathbf{a}_p\|^2
\ =\ \frac{c^2}{(2\gamma-1)}\left(p^2+\frac{(1-p)^2}{N-1}\right)
\qquad\mbox{for } 1/2<\gamma\leq 1.
\end{equation*}
Furthermore, since $Sp(W)\setminus\{1\}=0$ with multiplicity $(N-1)$,
conditions $\lambda^{*}<1-(2c)^{-1}$ or $\lambda^{*}=1-(2c)^{-1}$ required in the
previous results when $\gamma=1$ correspond to the conditions $c>1/2$
or $c=1/2$ and, setting
$$A_p\ :=\ UV^{\top}(UV^{\top})^{\top}\ =\ (I-\mathbf{1}\mathbf{a}_p^{\top})
(I-\mathbf{a}_p\mathbf{1}^{\top})
\ =\
I+
\|\mathbf{a}_p\|^2\mathbf{1}\mathbf{1}^{\top}
-(\mathbf{1}\mathbf{a}_p^{\top}+\mathbf{a}_p \mathbf{1}^{\top}),
$$
we have
\begin{itemize}
\item[(a)] $\widehat{\Sigma}_{\gamma}=\frac{c}{2}A_p$ for $1/2<\gamma<1$;

\item[(b)] $\widehat{\Sigma}_1=\frac{c^2}{2c-1}A_p$, for $\gamma=1$ and $c>1/2$;

\item[(c)] $\widehat{\Sigma}_1^*=\frac{1}{4} A_p$, for $\gamma=1$ and $c=1/2$.
\end{itemize}

\indent In order to highlight the role of the ``special vertex'' in
the synchronization of the system, let us set the initial state of the
stochastic processes at the vertices as follows: $Z^1_0=z_1$ for the
``special vertex'' and $Z^2_0=\dots=Z^N_0=z_2$ for the other vertices
of the graph, with $z_1\neq z_2$.  This may represent a situation in
which initially a ``special subject'' has an inclination $z_1$ that is
different from the rest of the population which is setted on another
inclination $z_2$.  Since $(\widetilde{Z}_n)_n$ is a martingale, we
have that $E[Z_{\infty}]=E[\widetilde{Z}_0]=N^{-1/2}\mathbf{v}_1^{\top}
\mathbf{Z}_0$, which in this case reduces to
$$ E[Z_{\infty}]=z_1p+z_2(1-p).$$ Then, the expected limiting
inclination, i.e. $E[Z_{\infty}]$, is strongly related to the
influence that the ``special vertex'' exercises on the rest of the
vertices (which is ruled by the parameter $p$).  For instance,
consider the following cases:
\begin{itemize}
\item[(i)] If $p\simeq 1$, then we have $E[Z_{\infty}]\simeq z_1$
  regardless the value of $z_2$; this reflects a situation in which
  the ``special vertex'' is very charismatic in the system and he
  leads the other elements to synchronize on average towards his
  initial inclination.
\item[(ii)] If $p=1/N$ with $N$ large , then we have
  $E[Z_{\infty}]\simeq z_2$ regardless the value of $z_1$; this
  reflects a situation in which the ``diversity'' of the ``special
  vertex'' is dispersed because of the large number of individuals in the
  population, and so the expected limiting inclination is close to the
  initial inclination of the majority of the system.
\end{itemize}

\section{Statistical Inference}
\label{section_applications}

First of all, we observe that by means of the central limit theorem
for $\widetilde{Z}_{n}=N^{-1/2}\,\mathbf{v}_1^{\top}\,\mathbf{Z}_{n}$
presented in Theorem~\ref{thm:asymptotics_Z_tilde}, it is possible to
construct asymptotic confidence intervals for $Z_{\infty}$, i.e. the
limit random variable at which all the stochastic processes
$\{(Z_{n,j})_n:\, 1\leq j\leq N\}$ converge.
Specifically, an asymptotic confidence interval for $Z_{\infty}$ with approximate level
$(1-\theta)$ is the following:
\begin{equation}\label{eq:confidence_interval_Z_inf}
CI_{1-\theta}(Z_{\infty})\ :=\
\left(\
\widetilde{Z}_{n}-
\widetilde{\sigma}_{\gamma}\sqrt{\widetilde{Z}_n(1-\widetilde{Z}_n)}
n^{-(\gamma-1/2)}
z_{\theta}\ ;\
\widetilde{Z}_{n}+
\widetilde{\sigma}_{\gamma}\sqrt{\widetilde{Z}_n(1-\widetilde{Z}_n)}
n^{-(\gamma-1/2)}
z_{\theta}\ \right)
\end{equation}
where $z_\theta$ is such that
${\mathcal N}(0,1)(z_\theta,+\infty)=\theta/2$.  \\ Note that, in order to
compute the above confidence interval, we need to know $\mathbf{v}_1$
(as well as $N$, $c$ and $\gamma$). Nevertheless, it is not required
to know the whole weighted adjacency matrix $W$. For example, for
doubly stochastic matrices, the vector $\mathbf{v}_1$ is known (see
Remark~\ref{rem:Z_tilde_doubly_stoch}).\\

\indent We now focus on the inferential problem of testing the
hypothesis that the network is characterized by a given weighted
adjacency matrix $W_0$, i.e. $H_0:\,W=W_0$, using the
multi-dimensional stochastic process $(\mathbf{Z}_n)_n$ observed at
the vertices. Since the distribution of $Z_{\infty}$ is unknown, we
propose a test statistics whose limit does not involve
$Z_{\infty}$. The parameters $N$, $c$ and $\gamma$ are again
considered known.  \\

\indent First, we need to introduce some notation.  Given
a $N\times N$ positive semi-definite matrix $\Sigma$ of rank $1\leq
r\leq(N-1)$ and having spectral decomposition $\Sigma=O\Lambda O^{\top}$
(more precisely, $\Lambda$ is the diagonal matrix containing the
eigenvalues of $\Sigma$ and the columns of $O$ form a corresponding
orthonormal basis of right eigenvectors), we denote by $L$ the
diagonal matrix such that
\[[L]_{ij}\ =\ \left\{
\begin{aligned}
&\lambda^{-1/2}_{j}\ &&\mbox{if } i=j\mbox{ and }\lambda_{j}>0,\\
&0\ &&otherwise,
\end{aligned}
\right.\]
and by $H$ the $r\times N$-matrix such that
\[[H]_{ij}\ =\ \left\{
\begin{aligned}
&1\ &&\mbox{if } i=j\mbox{ and }1\leq i\leq r,\\
&0\ &&otherwise.
\end{aligned}
\right.\]
Then:
\begin{itemize}
\item[(a)] when $1/2<\gamma<1$, take
  $\Sigma=\widehat{\Sigma}_{\gamma}$ with rank $r=(N-1)$ and set
  $O_{\gamma}=O$, $L_{\gamma}=L$, $H_{\gamma}=H$ and
  $M_{\gamma}=H_{\gamma}L_{\gamma}O_{\gamma}^{\top}$;

\item[(b)] when $\gamma=1$ and $\lambda^{*}<1-(2c)^{-1}$, take
  $\Sigma=\widehat{\Sigma}_1$ with rank $r=(N-1)$ and set
  $O_{1}=O$, $L_{1}=L$, $H_{1}=H$ and $M_1=H_{1}L_{1}O_{1}^{\top}$;

\item[(c)] when $\gamma=1$ and $\lambda^{*}=1-(2c)^{-1}$, take
  $\Sigma=\widehat{\Sigma}_1^*$ with rank $r$ equal to the cardinality
  $m^*$ of the set
 $$ A^{*}\ =\ \left\{\ \lambda_j\in Sp(W)\ :\ {\mathcal
  Re}(\lambda_j)=1-(2c)^{-1}\ \right\},$$ defined
  in~\eqref{def:A_star} and set $O_1^*=O$, $L_1^*=L$, $H_1^*=H$ and
  $M_1^*=H_{1}^*L_{1}^*(O_{1}^*)^{\top}$.
\end{itemize}

\indent Fixed the weighted adjacency matrix assumed under $H_0$,
i.e. $W_0$, we can compute $\mathbf{v_1}$, $V$ and $U$ as defined in
Section~\ref{section_model}. Hence, we can obtain under $H_0$ the real
process
$\widetilde{Z}_{n}=N^{-1/2}\,(\mathbf{v}_1^{\top}\,\mathbf{Z}_{n})$ and
the multi-dimensional process
$\widehat{\mathbf{Z}}_{n}=(I-N^{-1/2}\mathbf{1}\mathbf{v}_1^{\top})\mathbf{Z}_{n}=
U\,V^{\top}\,\mathbf{Z}_{n}$.
Then, using \eqref{eq:conv_Z_tilde}, \eqref{eq:CLT_Z_hat_gamma},
\eqref{eq:CLT_Z_hat_less} and applying
Lemma~\ref{prop:gaussian_vector}, we have under $H_0$ that
\begin{itemize}
\item[(a)] for $1/2<\gamma<1$,
\begin{equation}\label{def:test_statistics_a}
\mathbf{T}_{\gamma,n}:=
n^{\gamma/2}\left[\widetilde{Z}_n(1-\widetilde{Z}_n)\right]^{-1/2}
M_{\gamma}\,\widehat{\mathbf{Z}}_n,
\end{equation}

\item[(b)] for $\gamma=1$ and $\lambda^{*}<1-(2c)^{-1}$,
\begin{equation}\label{def:test_statistics_b}
\mathbf{T}_{1,n}:= n^{1/2}\left[\widetilde{Z}_n(1-\widetilde{Z}_n)\right]^{-1/2}
M_1\,\widehat{\mathbf{Z}}_n,
\end{equation}
\end{itemize}
are asympto\-ti\-cally normal whose covariance matrix is the
$(N-1)\times(N-1)$ identity matrix.  Hence, both the test statistics
$\|\mathbf{T}_{\gamma,n}\|^2$ and $\|\mathbf{T}_{1,n}\|^2$ are
asympto\-ti\-cally chi-squared distributed with $(N-1)$ degrees of
freedom.  In the case (c), i.e. $\gamma=1$ and
$\lambda^{*}=1-(2c)^{-1}$, using \eqref{eq:conv_Z_tilde},
\eqref{eq:CLT_Z_hat_equal} and applying
Lemma~\ref{prop:gaussian_vector}, we have under $H_0$ that
\begin{equation}\label{def:test_statistics_c}
\mathbf{T}_{1,n}^*:=
\sqrt{\frac{n}{\ln(n)}}\left[\widetilde{Z}_n(1-\widetilde{Z}_n)\right]^{-1/2}
M_1^*
\widehat{\mathbf{Z}}_n
\end{equation}
is asympto\-ti\-cally normal whose covariance matrix is the
$m^* \times m^*$ identity matrix and hence the test statistics
$\|\mathbf{T}_{1,n}^{*}\|^2$ is asympto\-ti\-cally chi-squared
distributed with $m^*$ degrees of freedom.  These results let us
construct asymptotic critical regions for testing any $W_0$. We now
apply these testing procedures to the meaningful examples of weighted
adjacency matrices considered in Section~\ref{section_examples}.

\subsection{``Mean-field'' interaction}\label{subsection_test_medione}

Consider the family of weighted adjacency matrices
$\{W_{\alpha};\alpha\in (0,1]\}$ defined in~\eqref{def:W_crimaldi}. It
  may be of interest to test whether the unknown parameter $\alpha$
  can be assumed to be equal to a specific value $\alpha_0\in (0,1]$,
    i.e.
$$
H_0:\, W=W_{\alpha_0}\qquad\mbox{ vs} \qquad
H_1:\, W=W_{\alpha}\ \mbox{for some }\alpha\in (0,1]\setminus\{\alpha_0\}.
$$
In this case, assuming $2c\alpha_0\geq1$ when $\gamma=1$, by the
  results presented in Subsection~\ref{subsection_medione}, using
  $\mathbf{v}_1$ and $U=V$ computed for $W_{\alpha_0}$, we have:
\begin{itemize}
\item[(a)] for $1/2<\gamma<1$,
$\mathbf{T}_{\gamma,n}=
n^{\gamma/2}\left[\widetilde{Z}_n(1-\widetilde{Z}_n)\right]^{-1/2}
\sqrt{\frac{2\alpha_0}{c}}\,U^{\top}\,\widehat{\mathbf{Z}}_n$;

\item[(b)] for $\gamma=1$ and $2c\alpha_0>1$,
$\mathbf{T}_{1,n}=n^{1/2}\left[\widetilde{Z}_n(1-\widetilde{Z}_n)\right]^{-1/2}
\frac{\sqrt{2c\alpha_0-1}}{c}\,U^{\top}\,\widehat{\mathbf{Z}}_n$;

\item[(c)] for $\gamma=1$ and $2c\alpha_0=1$,
  $\mathbf{T}_{1,n}^{*}=
\sqrt{\frac{n}{\ln(n)}}\left[\widetilde{Z}_n(1-\widetilde{Z}_n)\right]^{-1/2}
  \frac{1}{c}\,U^{\top}\,\widehat{\mathbf{Z}}_n$,
\end{itemize}
where $\widetilde{Z}_{n}=N^{-1}\mathbf{1}^{\top}\mathbf{Z}_{n}$ and
$\widehat{\mathbf{Z}}_{n}=
\left(I-N^{-1}\mathbf{1}\mathbf{1}^{\top}\right)\mathbf{Z}_n$. Under
$H_0$ we have
$\|\mathbf{T}_{\gamma,n}\|^2,\|\mathbf{T}_{1,n}\|^2,\|\mathbf{T}_{1,n}^*\|^2
\stackrel{d}\sim\chi^2_{N-1}$.  Concerning the distribution of the
test statistics for $\alpha\neq \alpha_0$, notice that the
eigenvectors of $W$ do not depend on $\alpha$ and so $U$ is the same
for any $\alpha$.  Therefore, for any fixed $\alpha\in
(0,1]\setminus\{\alpha_0\}$, under the hypothesis
  $\{W=W_{\alpha}\}\subset H_1$ we have
\begin{itemize}
\item[(a)] for $1/2<\gamma<1$, $\|\mathbf{T}_{\gamma,n}\|^2
\stackrel{d}\sim(\frac{\alpha_0}{\alpha})\chi^2_{N-1}$;

\item[(b)] for $\gamma=1$ and $2c\alpha_0>1$, $\|\mathbf{T}_{1,n}\|^2
\stackrel{d}\sim(\frac{2c\alpha_0-1}{2c\alpha-1})\chi^2_{N-1}$ if $2c\alpha>1$,
and $\|\mathbf{T}_{1,n}\|^2\stackrel{P}\rightarrow +\infty$ if $2c\alpha=1$;

\item[(c)] for $\gamma=1$ and $2c\alpha_0=1$, $\|\mathbf{T}_{1,n}^*\|^2
\stackrel{P}\rightarrow 0$ for $2c\alpha>1$.
\end{itemize}

\subsection{``Cycle'' interaction}
\label{subsection_test_ciclo}

We could test whether the weighted adjacency matrix is the one, say
$W_0$, defined in~\eqref{def:W_ciclo}. Then, we consider the following
hypothesis test:
$$
H_0:\, W=W_0\qquad\mbox{ vs} \qquad
H_1:\, W\neq W_0.
$$
Once obtained the eigen-structure of $\widehat{\Sigma}_{\gamma}$,
$\widehat{\Sigma}_1$ and $\widehat{\Sigma}_1^*$, we can define
$\mathbf{T}_{\gamma,n}$, $\mathbf{T}_{1,n}$ and $\mathbf{T}_{1,n}^*$
as in~\eqref{def:test_statistics_a},~\eqref{def:test_statistics_b}
and~\eqref{def:test_statistics_c}, respectively, and under $H_0$ we
have that
\begin{itemize}
\item[(a)] for $1/2<\gamma<1$,
$\|\mathbf{T}_{\gamma,n}\|^2\stackrel{d}\sim\chi^2_{N-1}$;

\item[(b)] for $\gamma=1$ and $2c(1-\cos(2\pi/N))>1$,
$\|\mathbf{T}_{1,n}\|^2\stackrel{d}\sim\chi^2_{N-1}$;

\item[(c)] for $\gamma=1$ and $2c(1-\cos(2\pi/N))=1$,
$\|\mathbf{T}_{1,n}^*\|^2\stackrel{d}\sim\chi^2_{2}$.
\end{itemize}

\subsection{``Special vertex'' case}
\label{subsection_test_caso_non_bistocastico}

We could test whether there is a ``special vertex'' in the network,
that is the weighted adjacency matrix is the one, say $W_p$, defined
in~\eqref{def:W_caso_non_bistocastico}, and so in this case the
considered hypothesis test is the following:
$$
H_0:\, W=W_p\qquad\mbox{ vs} \qquad
H_1:\, W\neq W_p.
$$
Note by Subsection~\ref{subsection_caso_non_bistocastico} that
$\widehat{\Sigma}_{\gamma}=\frac{c}{2}A_p$,
$\widehat{\Sigma}_1=\frac{c^2}{2c-1}A_p$ and
$\widehat{\Sigma}_1^*=\frac{1}{4} A_p$,
where $A_p=(I-\mathbf{1}\mathbf{a}_p^{\top})(I-\mathbf{a}_p\mathbf{1}^{\top})$.
Hence, since $A_p$ has rank $(N-1)$, we have under $H_0$ that
$\|\mathbf{T}_{\gamma,n}\|^2$,
$\|\mathbf{T}_{1,n}\|^2$ and $\|\mathbf{T}_{1,n}^*\|^2$
defined as in~\eqref{def:test_statistics_a},~\eqref{def:test_statistics_b}
and~\eqref{def:test_statistics_c}
are all asympto\-ti\-cally chi-squared distributed with $(N-1)$
degrees of freedom.

\section{Variants}\label{section_variants}

We can consider the following two variants.

\subsection{The case of a ``forcing input''}

As in \cite{cri-dai-lou-min}, we can consider the following variant:

\begin{equation}\label{eq:dynamic-input}
\mathbf{Z}_{n+1}\ = \left(1-r_n\right)\mathbf{Z}_{n}\ +\ r_n\left[
  \rho \mathbf{X}_{n+1}+(1-\rho)q\mathbf{1}\right],
\end{equation}
where $E[\mathbf{X}_{n}|\mathcal{F}_{n-1}]=W^{\top}\,\mathbf{Z}_{n-1}$,
$\rho\in [0,1[$ and $q\in [0,1]$. The assumptions on $W$ and $(r_n)$
    are the same as in the previous sections. (Here we exclude the
    case $\rho=1$ since it corresponds to the model studied in the
    previous sections.)  \\

With the same notation as before, we consider the decomposition
\eqref{eq:decomposition_Z}. In particular, setting ${\widetilde Z}_n=
N^{-1/2}{\mathbf v}_1^{\top}{\mathbf Z}_n$, we obtain the dynamics
\begin{equation*}
\widetilde{Z}_{n+1}-\widetilde{Z}_{n}\ =\
-(1-\rho)r_n(\widetilde{Z}_{n}-q)
+\rho r_n N^{-1/2} \left(\mathbf{v}_1^{\top}\Delta\mathbf{M}_{n+1}\right),
\end{equation*}
where $\Delta
\mathbf{M}_{n+1}=\mathbf{X}_{n+1}-W^{\top}\,\mathbf{Z}_{n}$. Therefore we
have
\begin{equation*}
\widetilde{Z}_{n+1}-q\ =\
\left(1-(1-\rho)r_n\right)(\widetilde{Z}_{n}-q)
+\rho r_n N^{-1/2} \left(\mathbf{v}_1^{\top}\Delta\mathbf{M}_{n+1}\right)
\end{equation*}
and so
\begin{equation*}
(\widetilde{Z}_{n+1}-q)^2\ =\
\left(1-2r_n(1-\rho)\right)(\widetilde{Z}_{n}-q)^2
+r_n^2
\left[(1-\rho)^2 (\widetilde{Z}_{n}-q)^2+
\rho^2 N^{-1} \left(\mathbf{v}_1^{\top}\Delta\mathbf{M}_{n+1}\right)^2
\right].
\end{equation*}
It follows
\begin{equation*}
\begin{split}
E\left[(\widetilde{Z}_{n+1}-q)^2|\mathcal{F}_n\right]&\!=\!
\left(1-2r_n(1-\rho)\right)(\widetilde{Z}_{n}-q)^2
\!+\!r_n^2
\left\{\!
(1-\rho)^2 (\widetilde{Z}_{n}-q)^2
\!+\!\frac{\rho^2}{N}
E\left[
\left(\mathbf{v}_1^{\top}\Delta\mathbf{M}_{n+1}\right)^2|\mathcal{F}_n
\right]
\!\right\}\\
&\leq (\widetilde{Z}_{n}-q)^2 + r_n^2\xi_n.
\end{split}
\end{equation*}
Since $(\xi_n)_n$ is a bounded sequence of ${\mathcal F}_n$-measurable
random variables and $\sum_n r_n^2<+\infty$ for $1/2<\gamma\leq 1$, we
have that $\left((\widetilde{Z}_{n}-q)^2\right)_n$ is a positive almost
supermartingale and so it converges almost surely (and also in mean
since it is bounded). On the other hand, from the above computations,
we also get
\begin{equation*}
E\left[(\widetilde{Z}_{n+1}-q)^2\right]=
\left(1-2r_n(1-\rho)\right)E\left[(\widetilde{Z}_{n}-q)^2\right]
+ r_n^2 E[\xi_n]
\end{equation*}
and so, since $\rho<1$, by \cite[Lemma A.1]{cri-dai-lou-min}, we can
conclude that $\lim_n E\left[ (\widetilde{Z}_{n}-q)^2 \right]=0$.
Therefore, we obtain
$$
\widetilde{Z}_n \stackrel{a.s.}\longrightarrow q.
$$ Regarding $\widehat{\mathbf{Z}}_{n}={\mathbf Z}_n-{\mathbf
  1}{\widetilde Z}_n =U\,V^{\top}\,\mathbf{Z}_{n}$, we can prove again that
it converges almost surely to $\mathbf{0}$. Indeed, if we set
$\mathbf{Z}_{V,n}=V^{\top}\widehat{\mathbf{Z}}_{n}=V^{\top}\mathbf{Z}_n$, that is
$\widehat{\mathbf{Z}}_{n}=U\mathbf{Z}_{V,n}$ by
\eqref{eq:relazioni-2}, we get from \eqref{eq:dynamic-input}
\begin{equation*}
\begin{split}
\mathbf{Z}_{V,n+1}&=
(1-r_n)\mathbf{Z}_{V,n}\
+\ r_n\left[\rho V^{\top}\Delta \mathbf{M}_{n+1}+\rho V^{\top}W^{\top}\mathbf{Z}_n
+(1-\rho)qV^{\top}\mathbf{1}\right]
\\
&=
(1-r_n)\mathbf{Z}_{V,n}\
+\ r_n\left[\rho V^{\top}\Delta \mathbf{M}_{n+1}+
\rho V^{\top}(\mathbf{u}_1\mathbf{v}_1^{\top}+ UDV^{\top})\mathbf{Z}_n
+(1-\rho)qN^{1/2}V^{\top}\mathbf{u}_1\right]
\\
&=
(1-r_n)\mathbf{Z}_{V,n}\ +\ r_n\rho D \mathbf{Z}_{V,n}
\ +\  r_n\rho V^{\top}\Delta \mathbf{M}_{n+1}
\\
&=
\left[I-r_n(I-\rho D)\right]\mathbf{Z}_{V,n}\
+\ r_n\rho V^{\top}\Delta \mathbf{M}_{n+1},
\end{split}
\end{equation*}
where $I$ here denotes the $(N-1)\times(N-1)$ identity matrix.
Arguing as in the proof of Theorem \ref{th:as_conv_Z_hat}, we can
obtain that $\left(\|\mathbf{Z}_{V,n}\|^2\right)_n$ is a positive
almost supermartingale which satisfies
$$
x_{n+1}\ \leq\ (1-2a^*r_n)x_n\ +\ C r_n^2
$$ with $x_n=E[\|\mathbf{Z}_{V,n}\|^2]$, $a^*=1-\rho{\mathcal
  R}e(\lambda^*)$ and $C$ a suitable constant.  Since $\rho{\mathcal
  R}e(\lambda^*)<1$, we have $a^*>0$, which implies $\lim_n x_n=0$ and
so $\|\mathbf{Z}_{V,n}\|^2\to 0$ almost surely, that is
$\mathbf{Z}_{V,n}\to \mathbf{0}$ almost surely.
\\

Summing up, also for the considered variant, we have an almost sure
synchronization, that is all the random variables $Z_{n,k}$ converge
almost surely to the same limit, but in this case the limit is the
constant ``forcing input'' $q$. It is interesting to observe that this
occurs for any weighted adjacency matrix $W$ satisfying the required
assumptions. It is also worthwhile to note that, in this case, for the
above computations, we do not need the condition ${\mathcal
  R}e(\lambda^*)<1$ since $\rho < 1$ automatically implies $\rho
{\mathcal R}e(\lambda^*) <1$ when ${\mathcal R}e(\lambda^*)\leq 1$.
\\

\indent We refer to \cite{cri-dai-lou-min} for some functional central
limit theorems in the case of $\rho<1$ and the mean-field interaction.

\subsection{The case of a reducible weighted adjacency matrix}
\label{subsection_reducible_W}

We now consider an extension of the theory presented in this paper to
the case of {\em reducible} weighted adjacency matrix (see
\cite{ale-ghi} for a similar approach to systems of interacting
generalized Friedman urns).  Denoting by $m$, with $1\leq m\leq N$,
the multiplicity of the eigenvalue 1 of $W$,
i.e. $\lambda_1=...=\lambda_m=1$, the reducible matrix $W$ can in
general be expressed as follows:
\begin{equation}\label{def:W}
W\ =\ \left[\begin{aligned}
&W_1&&0&&...&&0&&W_{1f}\\
&0&&W_2&&...&&0&&W_{2f}\\
&:&&...&&...&&...&&...\\
&0&&0&&...&&W_m&&W_{mf}\\
&0&&0&&...&&0&&W_f
\end{aligned}
\right],
\end{equation}

where
\begin{itemize}
\item[(i)] $\{W_j;1\leq j\leq m\}$ are irreducible $n_j\times n_j$-matrices
  with $\lambda_{\max}(W_j)=1$;

\item[(ii)] (if exists) $W_f$ is a $n_f\times n_f$-matrix with
  $\lambda_{\max}(W_f)<1$;

\item[(iii)] (if exists) $\{W_{j,f};1\leq j\leq m\}$ are $n_j\times
  n_f$-matrices.
\end{itemize}
Naturally, $n_f+\sum^m_{j=1}n_j=N$.  The structure of $W$ given
in~\eqref{def:W} leads to a natural decomposition of the graph in
different sub-graphs $\{G_j;1\leq j\leq m\}$ associated to the sub-matrices
$\{W_j;1\leq j\leq m\}$ and $G_f$ associated to $W_f$.

Notice in~\eqref{def:W} that the vertices in $G_j$ are not influenced
by the vertices of the rest of the network, and hence the dynamics of
the corresponding processes can be fully established by considering
only the irreducible sub-matrix $W_j$ (see~\cite{ale-ghi} for further
details).  Then, applying the results presented in this paper to each
sub-graph $G_j$, it is possible to show that all the processes
positioned at the vertices in the same $G_j$ synchronize, that is they
all converge almost surely to the same limit.\\ \indent Concerning
$G_f$, the weighted adjacency matrix in~\eqref{def:W} shows that their
vertices are influenced by the vertices in $\{G_j;1\leq j\leq m\}$.
Specifically, applying similar arguments to the one adopted in this
paper, it is possible to establish that the processes in $G_f$
converge almost surely to convex combinations of the limits of the
processes in $\{G_j;1\leq j\leq m\}$, where the weights of such
combinations are related to the matrices $\{W_{jf};1\leq
j\leq m\}$.
\bigskip

\begin{center}
{\bf Acknowledgments}
\end{center}

\noindent Irene Crimaldi and Andrea Ghiglietti are members of the
Italian group ``Gruppo Nazionale per l'Analisi Matematica, la
Probabilit\`a e le loro Applicazioni (GNAMPA)'' of the Italian
Institute ``Istituto Nazionale di Alta Matematica (INdAM)''.\\ Giacomo
Aletti is a member of the Italian group ``Gruppo Nazionale per il
Calcolo Scientifico (GNCS)'' of the Italian Institute ``Istituto
Nazionale di Alta Matematica (INdAM)''.\\ Irene Crimaldi acknowledges
support from CNR PNR Project ``CRISIS Lab''.

\appendix\label{appendix}

\begin{center}
\huge{{\bf Appendix}}
\end{center}

\section{Some technical results}

In all the sequel, given $(a_n), (b_n)$ two sequences of real numbers
with $b_n\geq 0$, the notation $a_n=O(b_n)$ means $|a_n|\leq C b_n$
for a suitable constant $C>0$ and $n$ large enough.  Therefore, if we
also have $a_n^{-1}=O(b_n^{-1})$, then $C^{\top}b_n\leq |a_n|\leq C b_n$ for
suitable constants $C,C'>0$ and $n$ large enough. Given $(z_n), (z'_n)$
two sequences of complex numbers, with $z'_n\neq 0$, the notation
$z_n\sim z z'_n$, with $z\neq 0$, means $\lim_n z_n/z'_n=z$ and the
notation $z_n=o(z'_n)$ means $\lim_n z_n/z'_n=0$.

\subsection{Asymptotic results for sums of complex numbers}

We start recalling Toeplitz lemma (see \cite{lin-ros}), from which we
get useful corollaries employed in our proofs.

\begin{lem}\label{toeplitz-lemma-real} (Toeplitz lemma)\\
Let $\{x_{n,k}:\, 1\leq k\leq k_n\}$ be a triangular array of real
numbers with $k_n\uparrow +\infty$ and such that
\begin{itemize}
\item[i)] $\lim_n x_{n,k}=0$ for each fixed $k$;
\item[ii)] $\lim_n \sum_{k=1}^{k_n} x_{n,k}=1$;
\item[iii)] $\sum_{k=1}^{k_n} |x_{n,k}|=O(1)$.
\end{itemize}
Let $(y_n)_n$ be a sequence of real numbers with $\lim_n
y_n=y\in{\mathbb R}$. Then, we have $\lim_n \sum_{k=1}^{k_n}
x_{n,k} y_k=y$.
\end{lem}

\begin{rem}\label{toeplitz-real-zero}
\rm If in the above lemma we replace condition ii) by $\lim_n
\sum_{k=1}^{k_n} x_{n,k}=0$, we get $\lim_n \sum_{k=1}^{k_n}
x_{n,k}y_k=0$. Indeed,  applying Lemma \ref{toeplitz-lemma-real} to
$\widetilde x_{n,k}=x_{n,k}-(k_n)^{-1}$, we find
\begin{equation*}
\lim_n\sum_{k=1}^{k_n} \left(x_{n,k}-\frac{1}{k_n}\right)y_k=
\lim_n\sum_{k=1}^{k_n} \widetilde x_{n,k} y_k=
y
\end{equation*}
Hence, since $\lim_n \sum_{k=1}^{k_n}y_k/k_n=y$ (again by Lemma
\ref{toeplitz-lemma-real}), we finally get
$\lim_n\sum_{k=1}^{k_n}x_{n,k}y_k= 0$.
\end{rem}

From Lemma \ref{toeplitz-lemma-real}, we can easily get the following
corollary:
\begin{cor}\label{toeplitz-cor-real}
Let $(x_n)_n,\, (x'_n)_n$ and $(c_n)_n$ be three sequences of real
numbers such that $x'_n>0,\, c_n\geq 0$, $x_n\sim x\, x'_n$ with $x\in
(0,+\infty)$ and $\lim_n c_n=0$. Suppose to have
$\lim_n c_n \sum_{k=1}^n x_k=s\in \{0,1\}$, then
$\lim_n c_n \sum_{k=1}^n x'_k = s/x$.
\end{cor}
\begin{proof} By assumption, taking $\epsilon\in (0,x)$, we have
$x_n>(x-\epsilon)x'_n>0$ for $n\geq \bar{n}$ with a suitable
  $\bar{n}$. Moreover, since $c_n\to 0$, we have $\lim_n
  c_n\sum_{k=1}^n x_k'=\lim_n c_n\sum_{k=\bar{n}}^n x_k'$. Therefore,
  without loss of generality, we can suppose $x_n>0$ for each
  $n$. Hence, if $s=1$, it is enough to apply Lemma
  \ref{toeplitz-lemma-real} with $x_{n,k}=c_n x_k $, $y_n=x'_n/x_n$,
  $y=x^{-1}$; if $s=0$, it is enough to apply Remark
  \ref{toeplitz-real-zero} to $x_{n,k}=c_n x_k$.
\end{proof}

The following lemma extends Toeplitz Lemma and Remark
\ref{toeplitz-real-zero} to complex numbers:
\begin{lem}\label{toeplitz-lemma-complex} (Generalized Toeplitz lemma)\\
Let $\{z_{n,k}:\, 1\leq k\leq k_n\}$ be a
triangular array of complex numbers such that
\begin{itemize}
\item[i)] $\lim_n z_{n,k}=0$ for each fixed $k$;
\item[ii)] $\lim_n \sum_{k=1}^{k_n} z_{n,k}=s\in\{0,1\}$;
\item[iii)] $\sum_{k=1}^{k_n} |z_{n,k}|=O(1)$.
\end{itemize}
Let $(w_n)_n$ be a sequence of complex numbers with $\lim_n
w_n=w\in{\mathbb C}$. Then, we have $\lim_n \sum_{k=1}^{k_n}
z_{n,k}w_k=s w$.
\end{lem}
\begin{proof} Set $z_{n,k}=a_{n,k}+ i b_{n,k}$, $w_n=c_n + i
d_n$ and $w=c + i d$. By assumption i), we have $\lim_n
a_{n,k}=0$ and $\lim_n b_{n,k}=0$, for each fixed $k$, and, by
assumption ii), we have  $\lim_n \sum_{k=1}^{k_n} a_{n,k}= s$ and $\lim_n
\sum_{k=1}^{k_n} b_{n,k}= 0$.  Applying Lemma
\ref{toeplitz-lemma-real} to $a_{n,k}$, we easily get $\lim_n
\sum_{k=1}^{k_n}a_{n,k} c_k= s c$ and $\lim_n
\sum_{k=1}^{k_n}a_{n,k} d_k= s d$. Then, applying Remark
\ref{toeplitz-real-zero} to $b_{n,k}$, we find
$\lim_n\sum_{k=1}^{k_n}b_{n,k}c_k= 0$ and
$\lim_n\sum_{k=1}^{k_n}b_{n,k} d_k= 0$. Therefore, we have
\begin{equation*}
\sum_{k=1}^{k_n} z_{n,k}w_k
=
\sum_{k=1}^{k_n}a_{n,k} c_k - \sum_{k=1}^{k_n}b_{n,k}d_k +
i \sum_{k=1}^{k_n}a_{n,k} d_k + \sum_{k=1}^{k_n}b_{n,k} c_k
\longrightarrow s(c+ i d) = s w.
\end{equation*}
\end{proof}

As before, from this lemma, we can easily get the following corollaries:
\begin{cor}\label{toeplitz-cor-complex}
Let $(z_n)_n,\, (v_n)_n$ and $(w_n)_n$ be three sequences of complex
numbers such that $\lim_n v_n=0$ and $\lim_n w_n=w\neq 0$. Set
$z'_n=z_n w_n$ and suppose to have $\lim_n v_n \sum_{k=1}^{n}
z_k=s\in\{0,1\}$ and $|v_n| \sum_{k=1}^{n} |z_k|=O(1)$ or,
  equivalently, $|v_n| \sum_{k=1}^{n} |z'_k|=O(1)$.  Then
  $\lim_n v_n \sum_{k=1}^{n} z'_k = s w$.
\end{cor}
\begin{proof} It is enough to apply Lemma
\ref{toeplitz-lemma-complex} to $z_{n,k}=v_n z_k$ and $w_n$. To this
purpose, note that, by assumption, taking $\epsilon\in (0,|w|)$ (note
that $|w|>0$ by assumption), we have $0<|w|-\epsilon \leq |w_n|\leq
|w| + \epsilon$ for $n\geq \bar{n}$ with $\bar n$ large
enough. Therefore, by the relation $z'_k=z_k w_k$, we can affirm that
$$
|v_n| \sum_{k=\bar n}^n \frac{|z'_k|}{|w|+\epsilon}\leq
|v_n| \sum_{k=\bar n}^{n} |z_k|\leq
|v_n| \sum_{k=\bar n}^{n} \frac{|z'_k|}{|w|-\epsilon}
$$
and so the two conditions $|v_n| \sum_{k=1}^{n} |z_k|=O(1)$ and
$|v_n| \sum_{k=1}^{n} |z'_k|=O(1)$ are equivalent since $|v_n|\to 0$.
\end{proof}

\begin{cor}\label{kronecker-complex} (Generalized Kronecker lemma)\\
Let $(v_n)$ and $(z_k)$ be two sequences of complex numbers such that
$$
v_n\neq 0,\quad \lim_n v_n=0,\quad
|v_n|\sum_{k=1}^n \left|\frac{1}{v_k}-\frac{1}{v_{k-1}}\right|=O(1)
$$
and $\sum_n z_n$ is convergent. Then
$$
\lim_n v_n \sum_{k=1}^n \frac{z_k}{v_k}=0.
$$
\end{cor}
\begin{proof}
Set $w_n=\sum_{k=n}^{+\infty} z_k$ and observe that, since $\sum_n
z_n$ is convergent, we have $\lim_n w_n=w=0$ and, moreover, we can write
\begin{equation*}
\begin{split}
v_n\sum_{k=1}^n \frac{z_k}{v_k}&=v_n\sum_{k=1}^n \frac{w_k-w_{k+1}}{v_k}\\
&=v_n\left[\sum_{k=2}^n \left(\frac{1}{v_k}-\frac{1}{v_{k-1}}\right)w_k
+\frac{w_1}{v_1}-\frac{w_{n+1}}{v_n}\right]\\
&=v_n\sum_{k=2}^n \left(\frac{1}{v_k}-\frac{1}{v_{k-1}}\right)w_k
+v_n\frac{w_1}{v_1}- w_{n+1}.
\end{split}
\end{equation*}
The second and the third term obviously converge to zero. In order to
prove that the first term converges to zero, it is enough to apply
Lemma \ref{toeplitz-lemma-complex} to $w_n$ and
$z_{n,k}=v_n\left(\frac{1}{v_k}-\frac{1}{v_{k-1}}\right)$.
\end{proof}

The above corollary is useful to get the following result for
complex random variables:

\begin{lem}\label{lemma-serie-rv}
Let ${\mathcal H}=({\mathcal H}_n)_n$ be an increasing filtration and
$(Y_n)$ a $\mathcal H$-adapted sequence of complex random variables
such that $E[Y_{n}|{\mathcal H}_{n-1}]\to Y$ almost surely. Moreover,
let $(c_n)$ be a sequence of strictly positive real numbers such that
$\sum_n E\left[|Y_n|^2\right]/c_n^2<+\infty$ and  let
$(v_n)$ be a sequence of complex numbers such that $v_n\neq 0$ and
\begin{equation}\label{cond-serie-rv-1}
\lim_n v_n=0,\quad
\lim_n v_n\sum_{k=1}^n \frac{1}{c_kv_k}= \eta\in{\mathbb C},
\end{equation}
\begin{equation}\label{cond-serie-rv-2}
|v_n|\sum_{k=1}^n\frac{1}{c_k|v_k|}=O(1),
\quad
|v_n| \sum_{k=1}^n \left|\frac{1}{v_k}-\frac{1}{v_{k-1}}\right|=O(1).
\end{equation}
Then $\lim_n v_n \sum_{k=1}^n Y_k/(c_kv_k)=\eta Y$.
\end{lem}
\begin{proof} Let $A$ be an event such that $P(A)=1$ and $\lim_n
E[Y_n|{\mathcal H}_{n-1}](\omega)=Y(\omega)$ for each $\omega\in
A$. Fix $\omega \in A$ and set $w_n=E[Y_n|{\mathcal H}_{n-1}](\omega)$
and $w=Y(\omega)$. If $\eta\neq 0$, applying Lemma
\ref{toeplitz-lemma-complex} to $z_{n,k}=v_n/(c_kv_k\eta)$, $s=1$ and $w_n$,
we obtain
$$
\lim_n v_n \sum_{k=1}^n \frac{E[Y_k|{\mathcal H}_{k-1}](\omega)}{c_kv_k\eta}=
Y(\omega).
$$
If $\eta=0$,  applying Lemma
\ref{toeplitz-lemma-complex} to $z_{n,k}=v_n/(c_kv_k)$, $s=0$ and $w_n$,
we obtain
$$
\lim_n v_n \sum_{k=1}^n \frac{E[Y_k|{\mathcal H}_{k-1}](\omega)}{c_kv_k}=
0.
$$
Therefore, for both cases, we have
$$
v_n \sum_{k=1}^n \frac{E[Y_k|{\mathcal H}_{k-1}]}{c_kv_k}
\stackrel{a.s.}\longrightarrow \eta Y.
$$

Now, consider the martingale $(M_n)$ defined by
$$
M_n=\sum_{k=1}^n \frac{Y_k-E[Y_k|{\mathcal H}_{k-1}]}{c_k}.
$$
It is bounded in $L^2$ since $\sum_{k=1}^n
\frac{E[|Y_k|^2]}{c_k^2}<+\infty$ by assumption and so it is almost surely
convergent, that means $$ \sum_{k}
\frac{Y_k(\omega)-E[Y_k|{\mathcal H}_{k-1}](\omega)}{c_k}<+\infty
$$ for $\omega\in B$ with $P(B)=1$. Therefore, fixing $\omega\in B$
and setting $z_k=\frac{Y_k(\omega)-E[Y_k|{\mathcal
      H}_{k-1}](\omega)}{c_k}$, by Corollay \ref{kronecker-complex},
we get
$$
v_n \sum_{k=1}^n
\frac{Y_k(\omega)-E[Y_k|{\mathcal H}_{k-1}](\omega)}{c_kv_k}\longrightarrow 0.
$$
In order to conclude, it is enough to observe that
$$
v_n \sum_{k=1}^n \frac{Y_k}{c_kv_k}=
v_n \sum_{k=1}^n \frac{Y_k-E[Y_k|{\mathcal H}_{k-1}]}{c_kv_k} +
v_n \sum_{k=1}^n \frac{E[Y_k|{\mathcal H}_{k-1}]}{c_kv_k}.
$$
\end{proof}

\begin{rem}\label{rem-serie-rv}
\rm It is useful to note that, whenever $(v_n)$ is a decreasing
sequence of positive real numbers (the case of the classical Kronecker
lemma), conditions \eqref{cond-serie-rv-1} obviously entail conditions
\eqref{cond-serie-rv-2}. Moreover,
$$
|v_n|\sum_{k=1}^n\frac{1}{c_k|v_k|}=O(1)\quad\mbox{and}\quad
c_n|v_n|\left|\frac{1}{v_n}-\frac{1}{v_{n-1}}\right|=O(1)
$$
imply conditions \eqref{cond-serie-rv-2}. Indeed, we have
$$
|v_n|\sum_{k=1}^n\left|\frac{1}{v_k}-\frac{1}{v_{k-1}}\right|
=
|v_n|\sum_{k=1}^n\frac{1}{c_k|v_k|}
c_k|v_k|\left|\frac{1}{v_k}-\frac{1}{v_{k-1}}\right|.
$$
\end{rem}

We conclude this subsection recalling the following well-known
relations for $a\in{\mathbb R}$:
\begin{equation}\label{relazione-nota}
\sum_{k=1}^n\frac{1}{k^{1-a}}=
\begin{cases}
& O(1) \quad\mbox{for } a<0,\\
& \log(n) + d_n \quad\mbox{for } a=0,\\
&a^{-1}\, n^a + d_n \quad\mbox{for } 0< a< 1,
\end{cases}
\end{equation}
where $(d_n)$ denotes a bounded sequence.

\subsection{Asymptotic results for products of complex numbers}
We now present the framework for the results of this subsection.
Fix $1/2<\gamma\leq 1$ and $c>0$, and consider a sequence $(r_n)_n$ of
real numbers such that $0<r_n<1$ and
\begin{equation}\label{eq:condition_r_n}
r_n\sim \frac{c}{n^{\gamma}}.
\end{equation}

Let $\alpha_1=a_1+i\,b_1\in{\mathbb C}$ and
$\alpha_2=a_2+i\,b_2\in{\mathbb C}$ with $a_1,a_2>0$.
Denote by $m_0\geq 2$ an integer such that $max(a_1,a_2)<r_m^{-1}$
for all $m\geq m_0$ and define for $n\geq m_0$ and $j=1,2$,
$$
p_{n,j}=\prod_{m=m_0}^n (1-\alpha_j r_m) \qquad\mbox{and}\qquad
\ell_{n,j}=p_{n,j}^{-1}.
$$

Then, inspired by the computation done in \cite{cri-dai-lou-min}, we can prove
the following technical results.

\begin{lem}\label{lemma-tecnico_1}
For $j=1,2$ and for any $\epsilon\in (0,1)$, we have that
\begin{equation}\label{affermazione1}
|p_{n,j}|\ =\ \begin{cases}
O\left(\exp \left[ -(1-\epsilon)\frac{ca_j }{1-\gamma}n^{1-\gamma} \right]\right)
& \mbox{for } 1/2<\gamma<1 \\
O\left(n^{- (1-\epsilon)ca_j}\right) & \mbox{for } \gamma=1
\end{cases}
\end{equation}
and
\begin{equation}\label{affermazione1-ell}
|\ell_{n,j}|\ =\ \begin{cases}
O\left(\exp \left[ (1+\epsilon)\frac{ca_j }{1-\gamma}n^{1-\gamma} \right]\right)
& \mbox{for } 1/2<\gamma<1 \\
O\left(n^{(1+\epsilon)ca_j}\right) & \mbox{for } \gamma=1.
\end{cases}
\end{equation}
Moreover, if we replace~\eqref{eq:condition_r_n} with the following
\begin{equation}\label{eq:condition_r_n_bis}
n^{\gamma}r_n-c\ =\  O\left(n^{-\gamma}\right),
\end{equation}
we have that
\begin{equation}\label{affermazione1_bis}
|p_{n,j}|\ =\ \begin{cases}
O\left(\exp \left[ -\frac{ca_j }{1-\gamma}n^{1-\gamma} \right]\right)
& \mbox{for } 1/2<\gamma<1 \\
O\left(n^{- ca_j}\right) & \mbox{for } \gamma=1
\end{cases}
\end{equation}
and
\begin{equation}\label{affermazione1_bis-ell}
|\ell_{n,j}|\ =\ \begin{cases}
O\left(\exp \left[ \frac{ca_j }{1-\gamma}n^{1-\gamma} \right]\right)
& \mbox{for } 1/2<\gamma<1 \\
O\left(n^{ca_j}\right) & \mbox{for } \gamma=1.
\end{cases}
\end{equation}
\end{lem}

\begin{proof}
Consider $j=1,2$. We can easily write $p_{n,j}=p^*_{n,j} q_{n,j}$, where
\begin{equation*}
p^*_{n,j}=\prod_{m=m_0}^n (1-a_j r_m)
\quad\mbox{and}\quad
q_{n,j}=\prod_{m=m_0}^n \left(1-i \frac{b_j r_m}{1-a_jr_m}\right).
\end{equation*}
\noindent
We now observe that
$$ |q_{n,j}|^2\ =\ \prod_{m=m_0}^n \left(1+\frac{b_j^2
  r_m^2}{(1-a_jr_m)^2}\right) \ =\ \exp\left[
  \sum_{m=m_0}^n\ln\left(1+\frac{b_j^2 r_m^2}{(1-a_jr_m)^2}\right)
  \right],$$
and using the inequalities $-x\leq \ln(1+x)\leq x$ for
$x\geq 0$, we have that
$$
\exp\left[-b_j^2 \sum_{m=m_0}^n\frac{r_m^2}{(1-a_jr_m)^2}
\right]
\leq
|q_{n,j}|^2\
\leq\
\exp\left[b_j^2 \sum_{m=m_0}^n\frac{r_m^2}{(1-a_jr_m)^2}
\right].
$$
Hence, since the series $\sum_{m}\frac{r_m^2}{(1-a_jr_m)^2}$ is
convergent for $1/2<\gamma\leq 1$, we have $p_{n,j}=O(|p^*_{n,j}|)$
and $\ell_{n,j}=O(|\ell^*_{n,j}|)$ with $\ell^*_{n,j}=1/p^*_{n,j}$.
Therefore, it is enough to study
$$
p^*_{n,j}=\exp\left(\sum_{k=m_0}^{n}\ln(1-a_jr_k)\right)
\quad\mbox{and}\quad
\ell^*_{n,j}=\exp\left(-\sum_{k=m_0}^{n}\ln(1-a_jr_k)\right).
$$
\noindent Recalling the inequalities $\ln(1-x)\leq -x$ and
$-\ln(1-x)\leq x+x^2$ for $0\leq x\leq 1/2$ and the fact that the series
$\sum_k r_k^2$ is convergent for $1/2<\gamma\leq 1$, we get
\begin{equation*}
\begin{split}
p^*_{n,j}&=\ O\left( \exp\left(-a_j\sum_{k=m_0}^{n}r_k\right)\ \right)\\
\ell^*_{n,j}&=\ O\left( \exp\left(a_j\sum_{k=m_0}^{n}r_k\right)\ \right).
\end{split}
\end{equation*}

\noindent We now take into account the decomposition
$\exp(a_j\sum_{k=m_0}^{n}r_k)= s^*_{n,j} t^*_{n,j}$, where
\begin{equation*}
s^*_{n,j}=\exp\left(a_jc\sum_{k=m_0}^{n}k^{-\gamma}\right)
\quad\mbox{and}\quad
t^*_{n,j}=\exp\left(a_j\sum_{k=m_0}^{n}(r_k-ck^{-\gamma})\right).
\end{equation*}
Now, since by condition~\eqref{eq:condition_r_n}, for any $\epsilon\in (0,1)$
we have $|r_k-ck^{-\gamma}|\leq\epsilon ck^{-\gamma}$ for $k$ large
enough (depending on $\epsilon$), we obtain
$$
p^*_{n,j}=O((s^*_{n,j})^{-(1-\epsilon)})\quad\mbox{and}\quad
\ell^*_{n,j}=O((s^*_{n,j})^{1+\epsilon}).
$$
Then, \eqref{affermazione1} and \eqref{affermazione1-ell} follow by
noticing that, by means of \eqref{relazione-nota}, we have
\begin{equation}\label{eq:relazione_nota_s}
s^*_{n,j}=
\begin{cases}
O\left(\exp \left( \frac{ a_j c}{1-\gamma} n^{1-\gamma} \right)\right)
& \mbox{if } 1/2\!<\!\gamma\!<\!1
\\[3pt]
O\left( n^{a_jc}\right) & \mbox{if } \gamma=1
\end{cases}
\;\mbox{and}\;
(s^*_{n,j})^{-1}=
\begin{cases}
O\left(\exp \left( \frac{ - a_j c}{1-\gamma} n^{1-\gamma} \right)\right)
& \mbox{if } 1/2\!<\!\gamma\!<1
\\[3pt]
O\left( n^{- a_jc}\right) & \mbox{if } \gamma=1.
\end{cases}
\end{equation}

Finally, by condition~\eqref{eq:condition_r_n_bis} and since the series
$\sum_{k} O(k^{-2\gamma})$ is convergent, we have $t^*_{n,j}=O(1)$
and $(t^*_{n,j})^{-1}=O(1)$, which imply $p^*_{n,j}=O((s^*_{n,j})^{-1})$ and
$\ell^*_{n,j}=O(s^*_{n,j})$.  Then,
result~\eqref{affermazione1_bis} follows by applying
\eqref{eq:relazione_nota_s}.
\end{proof}

\begin{lem}\label{lemma-tecnico_2}
We have that
\begin{equation}\label{affermazione2}
\lim_n\,
n^{\gamma} p_{n,1} p_{n,2}\sum_{k=m_0}^{n} r_k^2 \,\ell_{k,1}\ell_{k,2}\ =\
\begin{cases}
\frac{c}{\alpha_1+\alpha_2 }\;&\mbox{if } 1/2<\gamma<1
\\[3pt]
\frac{c^2}{c(\alpha_1+\alpha_2)-1}\;&\mbox{if } \gamma=1,\ c(a_1+a_2)>1
\end{cases}
\end{equation}
and, for any $u\geq 1$, when $1/2<\gamma<1$ or when $\gamma=1$ and
$uc(a_1+a_2)>(2u-1)$, we have
\begin{equation}\label{affermazione3}
|p_{n,1}|^u\,|p_{n,2}|^u\sum_{k=m_0}^{n} r_k^{2u}\, |\ell_{k,1}|^u\,|\ell_{k,2}|^u
=O\left(n^{-\gamma (2u-1)}\right).
\end{equation}
\end{lem}

\begin{proof} Let us start with observing that
relations \eqref{affermazione1} imply in particular
\begin{equation}\label{lim-zero}
\lim_n n^{\gamma} |p_{n,1}|\,|p_{n,2}|=0.
\end{equation}
Indeed, this fact follows immediately for $1/2<\gamma<1$ and, for $\gamma
=1$ one has to note that, since we assume $c(a_1+a_2)>1$, we can
choose $\epsilon$ small enough so that $c(1-\epsilon)(a_1+a_2) > 1$.
Now, fix $k\geq 2$ and let us define the following quantity
\[\begin{aligned}
D_{\gamma,k}\ &&=&\ \frac{1}{k^{\gamma}} \ell_{k,1}\ell_{k,2} -
\frac{1}{(k-1)^{\gamma}}\ell_{k-1,1}\ell_{k-1,2}\\
&&=&\
\left(\frac{1}{k^{\gamma}} - \frac{1}{(k-1)^{\gamma}} \right)
\ell_{k-1,1}\ell_{k-1,2}
+\ \frac{1}{k^{\gamma}}\left(\ell_{k,1}\ell_{k,2} -
\ell_{k-1,1}\ell_{k-1,2}\right)\\
&&=&\
\ell_{k,1}\ell_{k,2}\left[
\left(\frac{1}{k^{\gamma}} - \frac{1}{(k-1)^{\gamma}} \right)
\frac{\ell_{k-1,1}\ell_{k-1,2}}{\ell_{k,1}\ell_{k,2}}
+\ \frac{1}{k^{\gamma}}
\left(1-\frac{\ell_{k-1,1}\ell_{k-1,2}}{\ell_{k,1}\ell_{k,2}}\right)
\right].
\end{aligned}\]
Then, we observe the following:
\begin{equation}\label{eq:conti_1}
\left(\frac{1}{k^{\gamma}} - \frac{1}{(k-1)^{\gamma}}\right)\ =\
-\frac{\gamma}{k^{1+\gamma}}
+O\left(\frac{1}{k^{2+\gamma}}\right)\ =\
-\frac{\gamma}{k^{1+\gamma}}
+o\left(\frac{1}{k^{1+\gamma}}\right)\quad\mbox{for } k\to +\infty
\end{equation}
and
\begin{equation}\label{eq:conti_2}
\frac{\ell_{k-1,1}\ell_{k-1,2}}{\ell_{k,1}\ell_{k,2}}\ =\
(1-\alpha_1 r_k)(1-\alpha_2 r_k)\ =\
1+\alpha_1\alpha_2 r_k^2-(\alpha_1+\alpha_2)r_k.
\end{equation}
Now, by using~\eqref{eq:conti_1} and \eqref{eq:conti_2} in the above
expression of $D_{\gamma,k}$, and recalling~\eqref{eq:condition_r_n},
we have for $k\to +\infty$
\begin{equation*}
\begin{split}
D_{\gamma,k}\ &=\
\ell_{k,1}\ell_{k,2}
\left[-\frac{\gamma}{k^{\gamma + 1}}(1-\alpha_1r_k)(1-\alpha_2r_k)
+\frac{1}{k^\gamma}
\left( -\alpha_1\alpha_2r_{k}^2 + (\alpha_1+\alpha_2)r_{k} \right)\right]
+o\left(\frac{\ell_{k,1}\ell_{k,2}}{k^{1+\gamma}}\right)
\\
&=\
\ell_{k,1}\ell_{k,2}
\left[\frac{r_k}{k^\gamma}(\alpha_1+\alpha_2)-\frac{\gamma}{k^{\gamma + 1}}\right]
+o\left(\frac{\ell_{k,1}\ell_{k,2}}{k^{1+\gamma}}\right)
\\
&=\
\begin{cases}
\!\frac{(\alpha_1+\alpha_2)r_{k}}{k^{\gamma}} \ell_{k,1}\ell_{k,2} +
o(r_k^2\,\ell_{k,1}\,\ell_{k,2}) &\mbox{if } 1/2<\gamma<1
\\[3pt]
\!\left(\!\frac{(\alpha_1+\alpha_2)r_{k}}{k}-
\frac{1}{k^2}\!\right)
\ell_{k,1}\ell_{k,2}  +
o(r_k^2\,\ell_{k,1}\,\ell_{k,2})
&\mbox{if } \gamma=1,\ c(\alpha_1+\alpha_2)\neq 1
\end{cases}
\end{split}
\end{equation*}
that is
\begin{equation}\label{conti}
D_{\gamma, k}\sim\
\begin{cases}
\!\frac{(\alpha_1+\alpha_2)}{c}r_k^2 \,\ell_{k,1}\ell_{k,2}
 &\mbox{if } 1/2<\gamma<1
\\[3pt]
\!\frac{c(\alpha_1+\alpha_2)-1}{c^2}r_{k}^2\, \ell_{k,1}\ell_{k,2}
&\mbox{if } \gamma=1,\ c(\alpha_1+\alpha_2)\neq 1.
\end{cases}
\end{equation}
Note that, when $\gamma=1$, the condition $c(a_1+a_2)>1$ implies that
$c(\alpha_1+\alpha_2)\neq 1$, that ensures $D_{1,k}\sim
r_k^2\ell_{k,1}\ell_{k,2}$. Now, we want to apply Corollary
\ref{toeplitz-cor-complex} with
$$
z_n=D_{\gamma,n},\qquad v_n=n^\gamma\,p_{n,1}p_{n,2},\qquad
w_n=\frac{r_n^2\ell_{n,1}\ell_{n,2}}{D_{\gamma,n}}, \qquad
w=
\begin{cases}
\frac{c}{(\alpha_1+\alpha_2)}\; &\mbox{if } 1/2<\gamma<1
\\[3pt]
\frac{c^2}{c(\alpha_1+\alpha_2)-1}\; &\mbox{if } \gamma=1,\; c(a_1+a_2)>1.
\end{cases}
$$
Indeed, $\lim_n v_n=0$ by \eqref{lim-zero}, $\lim_n w_n=w\neq 0$
by \eqref{conti},
$$
v_n\sum_{k=m_0}^n z_k= n^{\gamma}p_{n,1}p_{n,2}
\sum_{k=m_0}^{n} D_{\gamma,k}=
n^{\gamma}p_{n,1}p_{n,2}
\left(\frac{\ell_{n,1}\ell_{n,2}}{n^\gamma}-
\frac{\ell_{m_0-1,1}\ell_{m_0-1,2}}{(m_0-1)^\gamma}
\right)\longrightarrow 1
$$ by \eqref{lim-zero} and
$z'_n=z_nw_n=r_n^2\ell_{n,1}\ell_{n,2}$. Finally, in order to apply
Corollary \ref{toeplitz-cor-complex}, it remains to prove that
$|v_n|\sum_{k=1}^n |z'_k|=O(1)$. In order to do this, we apply
Corollary \ref{toeplitz-cor-real} to
$$
x_n=\frac{1}{n^{\gamma}} |\ell_{n,1}|\,|\ell_{n,2}| -
\frac{1}{(n-1)^{\gamma}}|\ell_{n-1,1}|\,|\ell_{n-1,2}|,\quad
x'_n=r_n^2|\ell_{n,1}|\,|\ell_{n,2}|>0,\quad
c_n=n^\gamma\,|p_{n,1}|\,|p_{n,2}|.
$$
Indeed, we have $\lim_n c_n\sum_{k=1}^n x_k=1$ and, since
\begin{equation}\label{eq:conti_2-modulo}
\frac{|\ell_{k-1,1}|\,|\ell_{k-1,2}|}{|\ell_{k,1}|\,|\ell_{k,2}|}\ =\
|1+\alpha_1\alpha_2 r_k^2-(\alpha_1+\alpha_2)r_k|
\ =\ 1-(a_1+a_2)r_k +O(r_k^2),
\end{equation}
by computations similar to the ones done above, we can obtain
\begin{equation}\label{conti-modulo}
x_n\sim\
\begin{cases}
\!\frac{(a_1+a_2)}{c}\, x_n'
 &\mbox{if } 1/2<\gamma<1
\\[3pt]
\!\frac{c(a_1+a_2)-1}{c^2}\,x_n'
&\mbox{if } \gamma=1,\ c(a_1+a_2)>1
\end{cases}
\end{equation}
where both constants belong to $(0,+\infty)$. Therefore
$c_n\sum_{k=1}^n x'_k$ converges and so it is bounded. Hence, we have
verified all the conditions required by Corollary
\ref{toeplitz-cor-complex} and so we can conclude that we have $\lim_n
v_n\sum_{k=1}^n z'_k=w$, i.e. \eqref{affermazione2}.
\\

\indent Regarding \eqref{affermazione3}, we have already considered
the case $u=1$, which is related to $c_n\sum_{k=1}^n x'_k$. Similarly,
in order to prove \eqref{affermazione3} for $u>1$, we use
\begin{equation}\label{eq:conti_2-modulo}
\frac{|\ell_{k-1,1}|^u\,|\ell_{k-1,2}|^u}{|\ell_{k,1}|^u\,|\ell_{k,2}|^u}\ =\
|1+\alpha_1\alpha_2 r_k^2-(\alpha_1+\alpha_2)r_k|^u
\ =\ 1-u(a_1+a_2)r_k +O(r_k^2)
\end{equation}
and apply Corollary
\ref{toeplitz-cor-real} again. Indeed, with computations similar to
the one done before, we obtain
\begin{equation*}
\frac{|\ell_{k,1}|^u\, |\ell_{k,2}|^u}{k^{\gamma(2u-1)}}  -
\frac{|\ell_{k-1,1}|^u |\ell_{k-1,2}|^u}{(k-1)^{\gamma(2u-1)}}
\sim
\begin{cases}
\frac{u(a_1+a_2)}{c^{2u-1}}\, r_{k}^{2u}\, |\ell_{k,1}|^u\, |\ell_{k,2}|^u
&\mbox{if } 1/2<\gamma<1
\\[3pt]
\frac{uc(a_1+a_2)-(2u-1)}{c^{2u}}\, r_{k}^{2u}\, |\ell_{k,1}|^u\, |\ell_{k,2}|^u
&\mbox{if } \gamma=1,\ uc(a_1+a_2)>2u-1,
\end{cases}
\end{equation*}
where both constants belong to $(0,+\infty)$, and so we have
\begin{equation*}
\begin{split}
&\lim_n n^{\gamma (2u-1)}\,|p_{n,1}|^u\, |p_{n,2}|^u
\sum_{k=m_0}^{n}  r_k^{2u}\, |\ell_{k,1}|^u\,  |\ell_{k,2}|^u
=\\
&C(\gamma,u)\, \lim_n n^{\gamma (2u-1)}\,|p_{n,1}|^u\, |p_{n,2}|^u
\sum_{k=m_0}^{n}
\frac{|\ell_{k,1}|^u\, |\ell_{k,2}|^u }{k^{\gamma(2u-1)}}  -
\frac{|\ell_{k-1,j}|^u\, |\ell_{k-1,2}|^u}{(k-1)^{\gamma(2u-1)}}
=
C(\gamma,u)
\end{split}
\end{equation*}
for a suitable constant $C(\gamma,u)\in (0,+\infty)$.
\end{proof}

\begin{rem}
\rm We note that, if $\gamma=1$ and \eqref{eq:condition_r_n_bis} holds,
then we can add to \eqref{affermazione3} the following:
\begin{equation}\label{affermazione3-bis}
|p_{n,1}|^u\,|p_{n,2}|^u\sum_{k=m_0}^{n} r_k^{2u}\, |\ell_{k,1}|^u\,|\ell_{k,2}|^u
=\begin{cases}
O\left(\ln(n)/n^{uc(a_1+a_2)}\right)\quad &\mbox{if } uc(a_1+a_2)=(2u-1)\\
O\left(n^{-uc(a_1+a_2)}\right)\quad &\mbox{if } uc(a_1+a_2)<(2u-1).
\end{cases}
\end{equation}
Indeed, by means
of~\eqref{affermazione1_bis} and \eqref{affermazione1_bis-ell} in
Lemma~\ref{lemma-tecnico_1}, we have
\begin{equation*}
\begin{split}
|p_{n,1}|^u\, |p_{n,2}|^u\sum_{k=1}^n r_k^{2u} |\ell_{k,1}|^u\,|\ell_{k,2}|^u &=
O\left(n^{-uc(a_1+a_2)}\right)\sum_{k=1}^n O\left(k^{uc(a_1+a_2)-2u}\right)\\
&=
O\left(n^{-uc(a_1+a_2)}\right)
\sum_{k=1}^n O\left(\frac{1}{k^{1-(uc(a_1+a_2)-2u+1)}}\right).
\end{split}
\end{equation*}
\end{rem}

\begin{lem}\label{lemma-tecnico_3}
Let $\gamma=1$, $c(a_1+a_2)=1$ and replace
condition~\eqref{eq:condition_r_n}
by~\eqref{eq:condition_r_n_bis}. Then, we have
\begin{equation}\label{affermazione2-log}
\lim_n\,
\frac{n}{\ln(n)}p_{n,1} p_{n,2}\sum_{k=m_0}^{n} r_k^2 \,\ell_{k,1}\ell_{k,2}
\ =\
\begin{cases}
0\;&\mbox{if } b_1+b_2\neq 0
\\[3pt]
c^2\;&\mbox{if } b_1+b_2 = 0
\end{cases}
\end{equation}
and
\begin{equation}\label{affermazione3-log}
|p_{n,1}|^u\,|p_{n,2}|^u\sum_{k=m_0}^{n} r_k^{2u}\, |\ell_{k,1}|^u\,|\ell_{k,2}|^u
=
\begin{cases}
O(\ln(n)/n)\quad&\mbox{for } u=1
\\[3pt]
O\left(n^{-u}\right)\quad&\mbox{for } u>1.
\end{cases}
\end{equation}
\end{lem}

\begin{proof}
First, note that~\eqref{affermazione2-log} for the case $b_1+b_2\neq
0$ can be established using the computations done for the proof of
Lemma~\ref{lemma-tecnico_2} with $\gamma=1$.  Indeed, we can apply
Corollary \ref{toeplitz-cor-complex} with
$$
z_n=D_{1,n},\quad
v_n=\frac{n}{\ln(n)}\,p_{n,1}p_{n,2},\quad
w_n=\frac{r_n^2\ell_{n,1}\ell_{n,2}}{D_{1,n}}, \quad
z_n'=z_nw_n=r_n^2\ell_{n,1}\ell_{n,2},\quad
w=\frac{c^2}{c(\alpha_1+\alpha_2)-1}
$$
In fact, by assumptions~\eqref{eq:condition_r_n_bis} and
$c(a_1+a_2)=1$, we have
\begin{equation}\label{lim-zero-log}
\lim_n \frac{n}{\ln(n)}\,|p_{n,1}|\, |p_{n,2}|=0
\end{equation}
since~\eqref{affermazione1_bis} in Lemma~\ref{lemma-tecnico_1} and,
moreover, we have $\lim_n w_n=w\neq 0$ by \eqref{conti} since
$c(\alpha_1+\alpha_2)\neq 1$, and $|v_n|\sum_{k=1}^n|z_k'|=O(1)$ by
\eqref{affermazione3-bis} with $u=1$ and, finally, we have $\lim_n
v_n\sum_{k=1}^n z_k=0$.  \\

\indent We now focus on the case $b_1+b_2=0$.  Fix $k\geq 2$ and let
us define the following quantity
\[\begin{aligned}
D_{\ln,k}\ &=\ \frac{\ln(k)}{k} \ell_{k,1}\ell_{k,2} -
\frac{\ln(k-1)}{k-1}\ell_{k-1,1}\ell_{k-1,2}\ \\
&=\
\left(\frac{\ln(k)}{k} - \frac{\ln(k-1)}{k-1} \right)
\ell_{k-1,1}\ell_{k-1,2}
+\ \frac{\ln(k)}{k}\left(\ell_{k,1}\ell_{k,2} -
\ell_{k-1,1}\ell_{k-1,2}\right)
\\
&=\
\ell_{k,1}\ell_{k,2}\left[
\left(\frac{\ln(k)}{k} - \frac{\ln(k-1)}{k-1} \right)
\frac{\ell_{k-1,1}\ell_{k-1,2}}{\ell_{k,1}\ell_{k,2}}
+\
\frac{\ln(k)}{k}
\left(1 -\frac{\ell_{k-1,1}\ell_{k-1,2}}{\ell_{k,1}\ell_{k,2}}\right)
\right]
\end{aligned}\]
We observe that for $k\to +\infty$
\begin{equation}\label{eq:conti_4}
\begin{aligned}
\left(\frac{\ln(k)}{k} - \frac{\ln(k-1)}{k-1}\right)\
&&=&\ -\frac{\ln(k)}{k(k-1)}\ -\frac{\ln(1-k^{-1})}{k-1}\\
&&=&\ -\frac{\ln(k)}{k^2}+\frac{1}{k^2} + O\left(\frac{\ln(k)}{k^3}\right)\\
&&=&\ -\frac{\ln(k)}{k^2}+\frac{1}{k^2} + o\left(\frac{1}{k^2}\right).
\end{aligned}
\end{equation}
Now, by using~\eqref{eq:conti_2} and~\eqref{eq:conti_4} in the
expression of $D_{\ln,k}$, and recalling~\eqref{eq:condition_r_n_bis},
we have that
\begin{equation*}
\begin{split}
D_{\ln,k}&=
\ell_{k,1}\ell_{k,2}
\left[
\left(-\frac{\ln(k)}{k^2}+\frac{1}{k^2}\right)(1-\alpha_1r_k)(1-\alpha_2r_k)
+\frac{\ln(k)}{k}
\left( -\alpha_1\alpha_2 r_{k}^2 + (\alpha_1+\alpha_2)r_{k} \right)\right]\\
&+\ o\left(\frac{\ell_{k,1}\ell_{k,2}}{k^2}\right)\\
&=\
\ell_{k,1}\ell_{k,2}
\left[\frac{r_k\ln(k)}{k}(\alpha_1+\alpha_2)-\frac{\ln(k)}{k^2}+
\frac{1}{k^2}\right]
+o\left(\frac{\ell_{k,1}\ell_{k,2}}{k^2}\right)\\
&=\
\left[\ln(k)\left(\!\frac{(\alpha_1+\alpha_2)r_{k}}{k}-\frac{1}{k^2}\!\right)
+\frac{1}{k^2}\right]
\ell_{k,1}\ell_{k,2}  + o\left(\frac{\ell_{k,1}\ell_{k,2}}{k^2}\right).
\end{split}
\end{equation*}
Then, since the equalities $c(a_1+a_2)=1$ and $b_1+b_2=0$ imply
$c(\alpha_1+\alpha_2)=1$, and recalling~\eqref{eq:condition_r_n_bis},
we obtain
\begin{equation}\label{conti_log}
D_{\ln,k}\ =\ \frac{1}{k^2}\ell_{k,1}\ell_{k,2}  +
o\left(\frac{\ell_{k,1}\ell_{k,2}}{k^2}\right)
\ =\ \frac{1}{k^2}\ell_{k,1}\ell_{k,2}  +
o(r_k^2\ell_{k,1}\ell_{k,2})\sim\
\frac{1}{c^2}r_k^2\ell_{k,1}\ell_{k,2}.
\end{equation}
Now, we want to apply Corollary \ref{toeplitz-cor-complex} with
$$
z_n=D_{\ln,n},\qquad v_n=\frac{n}{\ln(n)}\,p_{n,1}p_{n,2},\qquad
w_n=\frac{r_n^2\ell_{n,1}\ell_{n,2}}{D_{\ln,n}}, \qquad
z'_n=z_nw_n=r_n^2\ell_{n,1}\ell_{n,2},\qquad
w=c^2.
$$ Indeed, $\lim_n v_n=0$ by \eqref{lim-zero-log}, $\lim_n w_n=w\neq
0$ by \eqref{conti_log}, $|v_n|\sum_{k=1}^n |z'_k|=O(1)$ by
\eqref{affermazione3-bis} (with $u=1$) since $c(a_1+a_2)=1$ by
assumption,
\begin{equation*}
\begin{split}
\lim_n\,v_n\sum_{k=m_0}^n z_k&=\lim_n\, \frac{n}{\ln(n)}p_{n,1}p_{n,2}
\sum_{k=m_0}^{n} D_{\ln,k}\\
&=
\frac{n}{\ln(n)}p_{n,1}p_{n,2}
\left(\frac{\ln(n)\,\ell_{n,1}\ell_{n,2}}{n}-
\frac{\ln(m_0-1)\ell_{m_0-1,1}\ell_{m_0-1,2}}{(m_0-1)}
\right)\longrightarrow 1
\end{split}
\end{equation*}
by \eqref{lim-zero-log}.
Hence, all the conditions required by Corollary
\ref{toeplitz-cor-complex} hold and so we can conclude that we have
$\lim_n v_n\sum_{k=1}^n z'_k=w$, i.e. \eqref{affermazione2-log} for
$b_1+b_2=0$.  \\

\indent Finally, relations \eqref{affermazione3-log} follows from
\eqref{affermazione3-bis} using the assumption that $c(a_1+a_2)=1$.
\end{proof}

\subsection{A result for Gaussian random vectors}

The following result is about the standardization of Gaussian random
vectors with singular covariance matrix.

\begin{lem}\label{prop:gaussian_vector}
Let $X$ be a random vector with distribution
$\mathcal{N}_N(\mathbf{0},\Sigma)$ and consider the spectral
decomposition $\Sigma=O\Lambda O^{\top}$ (more precisely, $\Lambda$ is the
diagonal matrix containing the eigenvalues of $\Sigma$ and the columns
of $O$ form a corresponding orthonormal basis of right eigenvectors).
Let $1\leq r<N$ be the rank of $\Sigma$, define the matrix $L$ as
follows
\[[L]_{ij}\ =\ \left\{
\begin{aligned}
&\lambda^{-1/2}_{i}\ &&\mbox{if } i=j\mbox{ and }\lambda_{i}>0,\\
&0\ &&\mbox{otherwise},
\end{aligned}
\right.\]
and denote by $H$ the $r\times N$-matrix such that
\[[H]_{ij}\ =\ \left\{
\begin{aligned}
&1\ &&\mbox{if } i=j\mbox{ and }1\leq i\leq r,\\
&0\ &&\mbox{otherwise}.
\end{aligned}
\right.\]
Then, setting $M=HLO^{\top}$ and $Y=MX$, the distribution of
$Y$ is $\mathcal{N}_r(\mathbf{0},I)$.
\end{lem}

\proof It is immediate to see that $Y$ is a Gaussian vector since it
is a linear transformation of the Gaussian vector $X$.  Then, the
result follows by noticing that
$$Cov(Y)\ =\ M\Sigma M^{\top}\ =
\ HL(O^{\top}\Sigma O)LH^{\top}\ =\ H(L\Lambda L)H^{\top} \ =\ I.$$
\endproof

\section{Stable convergence and its variants}

 We recall here some basic definitions and results. For more details,
 we refer the reader to \cite{cri-let-pra-2007, hall-1980} and the
 references therein.\\

\indent Let $(\Omega, {\mathcal A}, P)$ be a probability space, and let
$S$ be a Polish space, endowed with its Borel $\sigma$-field. A {\em
  kernel} on $S$, or a random probability measure on $S$, is a
collection $K=\{K(\omega):\, \omega\in\Omega\}$ of probability
measures on the Borel $\sigma$-field of $S$ such that, for each
bounded Borel real function $f$ on $S$, the map
$$
\omega\mapsto
K\!f(\omega)=\int f (x)\, K(\omega)(dx)
$$
is $\mathcal A$-measurable. Given a sub-$\sigma$-field $\mathcal H$ of
$\mathcal A$, a kernel $K$ is said $\mathcal H$-measurable if all the
above random variables $K\!f$ are $\mathcal H$-measurable.\\

\indent On $(\Omega, {\mathcal A},P)$, let $(Y_n)$ be a sequence of
$S$-valued random variables, let $\mathcal H$ be a sub-$\sigma$-field
of $\mathcal A$, and let $K$ be a $\mathcal H$-measurable kernel on
$S$. Then we say that $Y_n$ converges {\em $\mathcal H$-stably} to
$K$, and we write $Y_n\stackrel{{\mathcal H}-stably}\longrightarrow
K$, if
$$
P(Y_n \in \cdot \,|\, H)\stackrel{weakly}\longrightarrow
E\left[K(\cdot)\,|\, H \right]
\qquad\hbox{for all } H\in{\mathcal H}\; \hbox{with } P(H) > 0.
$$
In the case when ${\mathcal H}={\mathcal A}$, we simply say that $Y_n$
converges {\em stably} to $K$ and we write
$Y_n\stackrel{stably}\longrightarrow K$. Clearly, if
$Y_n\stackrel{{\mathcal H}-stably}\longrightarrow K$, then $Y_n$
converges in distribution to the probability distribution
$E[K(\cdot)]$. Moreover, the $\mathcal H$-stable convergence of $Y_n$
to $K$ can be stated in terms of the following convergence of
conditional expectations:
\begin{equation}\label{def-stable}
E[f(Y_n)\,|\, {\mathcal H}]\stackrel{\sigma(L^1,\, L^{\infty})}\longrightarrow
K\!f
\end{equation}
for each bounded continuous real function $f$ on $S$. \\

\indent In \cite{cri-let-pra-2007} the notion of $\mathcal H$-stable
convergence is firstly generalized in a natural way replacing in
(\ref{def-stable}) the single sub-$\sigma$-field $\mathcal H$ by a
collection ${\mathcal G}=({\mathcal G}_n)$ (called conditioning system) of
sub-$\sigma$-fields of $\mathcal A$ and then it is strengthened by
substituting the convergence in $\sigma(L^1,L^{\infty})$ by the one in
probability (i.e. in $L^1$, since $f$ is bounded). Hence, according to
\cite{cri-let-pra-2007}, we say that $Y_n$ converges to $K$ {\em
  stably in the strong sense}, with respect to ${\mathcal G}=({\mathcal
  G}_n)$, if
\begin{equation}\label{def-stable-strong}
E\left[f(Y_n)\,|\,{\mathcal G}_n\right]\stackrel{P}\longrightarrow K\!f
\end{equation}
for each bounded continuous real function $f$ on $S$.\\

\indent Finally, a strengthening of the stable convergence in the
strong sense can be naturally obtained if in (\ref{def-stable-strong}) we
replace the convergence in probability by the almost sure convergence:
given a conditioning system ${\mathcal G}=({\mathcal G}_n)$, we say that $Y_n$
converges to $K$ in the sense of the {\em almost sure conditional
  convergence}, with respect to ${\mathcal G}$, if
\begin{equation}\label{def-as-cond}
E\left[f(Y_n)\,|\,{\mathcal G}_n\right]\stackrel{a.s.}\longrightarrow K\!f
\end{equation}
for each bounded continuous real function $f$ on $S$. Evidently, this
last type of convergence can be reformulated using the conditional
distributions. Indeed, if $K_n$ denotes a version of the conditional
distribution of $Y_n$ given ${\mathcal G}_n$, then the random variable
$K_n\!f$ is a version of the conditional expectation
$E\left[f(Y_n)|{\mathcal G}_n\right]$ and so we can say that $Y_n$ converges
to $K$ in the sense of the almost sure conditional convergence, with
respect to $\mathcal F$, if, for almost every $\omega$ in $\Omega$,
the probability measure $K_n(\omega)$ converges weakly to
$K(\omega)$. The almost sure conditional convergence has been
introduced in \cite{crimaldi-2009} and, subsequently, employed by
others in the urn model literature (e.g. \cite{aletti-2009, z}).
\\

We now conclude this section with some convergence results that we
need in our proofs. \\

From \cite[Proposition~3.1]{cri-pra}), we can get the following
result.

\begin{theo}\label{thm:triangular}
Let $({\mathbf T}_{n,k})_{n\geq 1, 1\leq k\leq k_n}$ be a triangular
array of $d$-dimensional real random vectors, such that, for each
fixed $n$, the finite sequence $({\mathbf T}_{n,k})_{1\leq k\leq k_n}$
is a martingale difference array with respect to a given filtration
$({\mathcal G}_{n,k})_{k\geq 0}$. Moreover, let $(t_n)_n$ be a
sequence of real numbers and assume that the following
conditions hold:
\begin{itemize}
\item[(c1)] ${\mathcal G}_{n,k}{\underline{\subset}} {\mathcal G}_{n+1,
  k}$ for each $n$ and $1\leq k\leq k_n$;
\item[(c2)] $\sum_{k=1}^{k_n} (t_n{\mathbf
  T}_{n,k})(t_n{\mathbf T}_{n,k})^{\top}=t_n^2\sum_{k=1}^{k_n} {\mathbf
  T}_{n,k}{\mathbf T}_{n,k}^{\top} \stackrel{P}\longrightarrow \Sigma$,
  where $\Sigma$ is a random positive semidefinite matrix;
\item[(c3)] $\sup_{1\leq k\leq k_n} |t_n{\mathbf T}_{n,k}|
\stackrel{L^1}\longrightarrow 0$.
\end{itemize}
Then $t_n\sum_{k=1}^{k_n}{\mathbf T}_{n,k}$ converges stably to the
Gaussian kernel ${\mathcal N}(0, \Sigma)$.
\end{theo}

The following result combines together a stable convergence and a
stable convergence in the strong sense.

\begin{theo}[{\cite[Lemma 1]{ber-cri-pra-rig}}]\label{blocco}
Suppose that $C_n$ and $D_n$ are $S$-valued random variables, that $M$
and $N$ are kernels on $S$, and that ${\mathcal G}=({\mathcal G}_n)_n$
is an (increasing) filtration satisfying for all $n$
$$
\sigma(C_n)\underline\subset{\mathcal G}_n\quad\hbox{and }\quad
\sigma(D_n)\underline\subset
\sigma\left({\textstyle\bigcup_n}{\mathcal G}_n\right)
$$

\noindent If $C_n$ stably converges to $M$ and $D_n$ converges to $N$
stably in the strong sense, with respect to $\mathcal G$, then
$$
[C_n, D_n]\stackrel{stably}\longrightarrow M \otimes N.
$$
(Here, $M\otimes N$ is the kernel on $S\times S$ such that $(M
\otimes N )(\omega) = M(\omega) \otimes N(\omega)$ for all $\omega$.)
\end{theo}

Given a conditioning system $\mathcal{G}=(\mathcal{G}_n)_n$, if
$\mathcal{U}$ is a sub-$\sigma$-field of $\mathcal{A}$ such that, for
each real integrable random variable $Y$, the conditional expectation
${\mathrm{E}}[Y\,|\,{\mathcal{G}}_n]$ converges almost surely to the
conditional expectation ${\rm \mathrm{E}}[Y\,|\,\mathcal{U}]$, then we
shall briefly say that $\mathcal{U}$ is an {\em asymptotic
  $\sigma$-field} for $\mathcal{G}$. In order that there exists an
asymptotic $\sigma$-field $\mathcal{U}$ for a given conditioning
system $\mathcal{G}$, it is obviously sufficient that the sequence
$(\mathcal{G}_n)_n$ is increasing or decreasing.  (Indeed we can take
$\mathcal{U}=\bigvee_n\mathcal{G}_n$ in the first case and
${\mathcal{U}}=\bigcap_n\mathcal{G}_n$ in the second one.)

\begin{theo}[{\cite[Theorem A.1]{crimaldi-2009}}]\label{fam_tri_vet_as_inf}
 On~$(\Omega,\mathcal{A},P)$,
  for each $n\geq 1$, let $(\mathcal{F}_{n,h})_{h\in{\mathbb N}}$ be a
  filtration and $(M_{n,h})_{h\in{\mathbb N}}$ a real martingale with respect
  to $({\mathcal{F}}_{n,h})_{h\in{\mathbb N}}$, with $M_{n,0}=0$, which
  converges in $L^1$ to a random variable $M_{n,\infty}$. Set

\begin{equation*}
X_{n,j}:=M_{n,j}-M_{n,j-1}\quad\hbox{for } j\geq 1,\quad
U_n:=\textstyle\sum_{j\geq 1}X_{n,j}^2,\quad
X_n^*:=\textstyle\sup_{j\geq 1}\;|X_{n,j}|.
\end{equation*}

Further, let $(k_n)_{n\geq 1}$ be a sequence of strictly positive
integers such that $k_nX_n^*\stackrel{a.s.}\to 0$ and let
$\mathcal{U}$ be a sub-$\sigma$-field which is asymptotic for the
conditioning system $\mathcal{G}$ defined by
$\mathcal{G}_n={\mathcal{F}}_{n,k_n}$. Assume that the sequence
$(X_n^*)_n$ is dominated in $L^1$ and that the sequence $(U_n)_n$
converges almost surely to a positive real random variable $U$ which
is measurable with respect to $\mathcal{U}$.  \\

Then, with respect to the conditioning
system $\mathcal{G}$, the sequence $(M_{n,\infty})_{n}$ converges to the
Gaussian kernel ${\mathcal{N}}(0,U)$ in the sense of the almost sure
conditional convergence.
\end{theo}

\end{document}